\documentclass[11pt,reqno,twoside]{amsart}

\usepackage{amssymb}
\usepackage{epstopdf}
\usepackage[asymmetric,top=3.5cm,bottom=3.3cm,left=3.1cm,right=3.1cm]{geometry}
\usepackage{booktabs} 
\usepackage{array} 
\usepackage{paralist} 
\usepackage{verbatim} 
\usepackage{graphicx,subfigure,threeparttable}

\usepackage{tabularx}
\usepackage{amsmath,amsfonts,amsthm,mathrsfs,cite}
\usepackage[usenames]{color}
\usepackage{bm}
\usepackage{tabularx}
\usepackage{longtable}
\usepackage{booktabs} 

\newtheorem{theorem}{Theorem}[section]
\newtheorem{lemma}{Lemma}[section]

\theoremstyle{definition}

\theoremstyle{remark}

\newtheorem{remark}{Remark}[section]
\numberwithin{equation}{section}

\usepackage{amsmath}

\allowdisplaybreaks[4]

\usepackage{hyperref}
\usepackage{xcolor}
\hypersetup{
  colorlinks,
  linkcolor={blue!80!black},
  urlcolor={blue!80!black},
  citecolor={blue!80!black}
}
\usepackage[capitalize,noabbrev]{cleveref}

\title[Time-domain direct sampling method]{A novel time-domain direct sampling approach for inverse scattering problems in acoustics}

\author{Yukun Guo}
\address{School of Mathematics, Harbin Institute of Technology, Harbin, P. R. China.}
\email{ykguo@hit.edu.cn}

\author{Hongjie Li}
\address{Yau Mathematical Sciences Center, Tsinghua University, Beijing, China;
Yanqi Lake Beijing Institute of Mathematical Sciences and Applications, Beijing, China}
\email{hongjieli@tsinghua.edu.cn}

\author{Xianchao Wang}
\address{School of Mathematics, Harbin Institute of Technology, Harbin, P. R. China.}
\email{xcwang90@gmail.com; xcwang@hit.edu.cn}

\date{} 
\begin{document}

\maketitle

\begin{abstract}
This work is concerned with an inverse scattering problem of determining unknown scatterers from time-dependent acoustic measurements. A novel time-domain direct sampling method is developed to efficiently determine both the locations and shapes of inhomogeneous media. In particular, our approach is very easy to implement since only cheap space-time integrations are involved in the evaluation of the imaging functionals.
Based on the  Fourier-Laplace transform, we establish an inherent connection between the time-domain and frequency-domain direct sampling method.
Moreover, rigorous theoretical justifications and numerical experiments are provided to verify the validity and feasibility of the proposed method.
\medskip

\end{abstract}

\section{Introduction}

In this paper, we consider an  inverse acoustic scattering problem of  reconstructing unknown  objects  by the associated time-dependent near-field measurements.
This type of inverse problem arises in a variety of practical applications and it attracts significant attentions in many fields of science and technology, such as sonar detection, medical imaging \cite{Liu2015}, non-destructive testing \cite{Ammari2013} and so on.

To begin with, we present the mathematical formulation of the time-dependent acoustic scattering problem in the inhomogeneous medium.
Let $c_0\in \mathbb{R}_+$ denote the background sound speed and $c(x)\in L^{\infty}(\mathbb{R}^d), d=2,3,$ be the sound speed in the inhomogeneous medium with $c(x)>0$. Suppose that $c(x)-c_0$ is compactly supported in $\Omega \in\mathbb{R}^d$ and $\mathbb{R}^d\backslash\overline{\Omega }$ is connected, where $\Omega$ is occupied by  the inhomogeneous medium.
We further assume that the incident wave $p^i(x,t)$ is a causal  signal (that is, $p^i(x,t)\equiv 0$ for $t\leq0$) emitted from a point $x\in \mathbb{R}^d\backslash \overline{\Omega }$.
Given the incident wave $p^i$ and the medium scatterer $\Omega$, the forward scattering problem is to find the scattered field $p^s(x,t)$ that satisfies the following wave equation
\begin{equation}\label{eq:main}
\begin{aligned}
&  c^{-2}(x)\partial_{tt}p^s(x,t)-\Delta p^s(x,t)=\left(c_0^{-2}-c^{-2}(x)\right) \partial_{tt} p^i(x,t), \quad (x,t)\in \mathbb{R}^d \times \mathbb{R}_+, \medskip\\
&    p^s(x,0)=\partial_t p^s(x,0)=0,   \qquad   x\in \mathbb{R}^d.
\end{aligned}
\end{equation}
The well-posedness of the forward scattering problem \eqref{eq:main} can be conveniently found in \cite{Guo2013, Cakoni2017}.
Throughout, let $\Gamma  \subset \mathbb{R}^d \backslash \overline{\Omega}$ be a closed or an open measurement surface such that the receivers are away from the target objects.
 Define the measured scattered  field by
\begin{equation*}
\Lambda:=\{ p^s(x,t): (x,t)\in \Gamma\times \mathbb{R}_+  \}.
\end{equation*}
 The inverse problem that we are concerned with is to determine the locations and shapes of the medium scatterers $\Omega$ by the knowledge of the near-field measurements $\Lambda$.
 Figure \ref{fig-model} provides a schematic illustration of the time-domain scattering problem in $\mathbb{R}^2$.
 It is noted that the unknown objects can be uniquely determined by
boundary measurements $\Lambda$ with a single incident wave when  the sound speed $c(x)$ is non-trapping and close to a constant \cite{Ma2023}.

 \begin{figure}
\centering
   \includegraphics[width=0.8\textwidth]{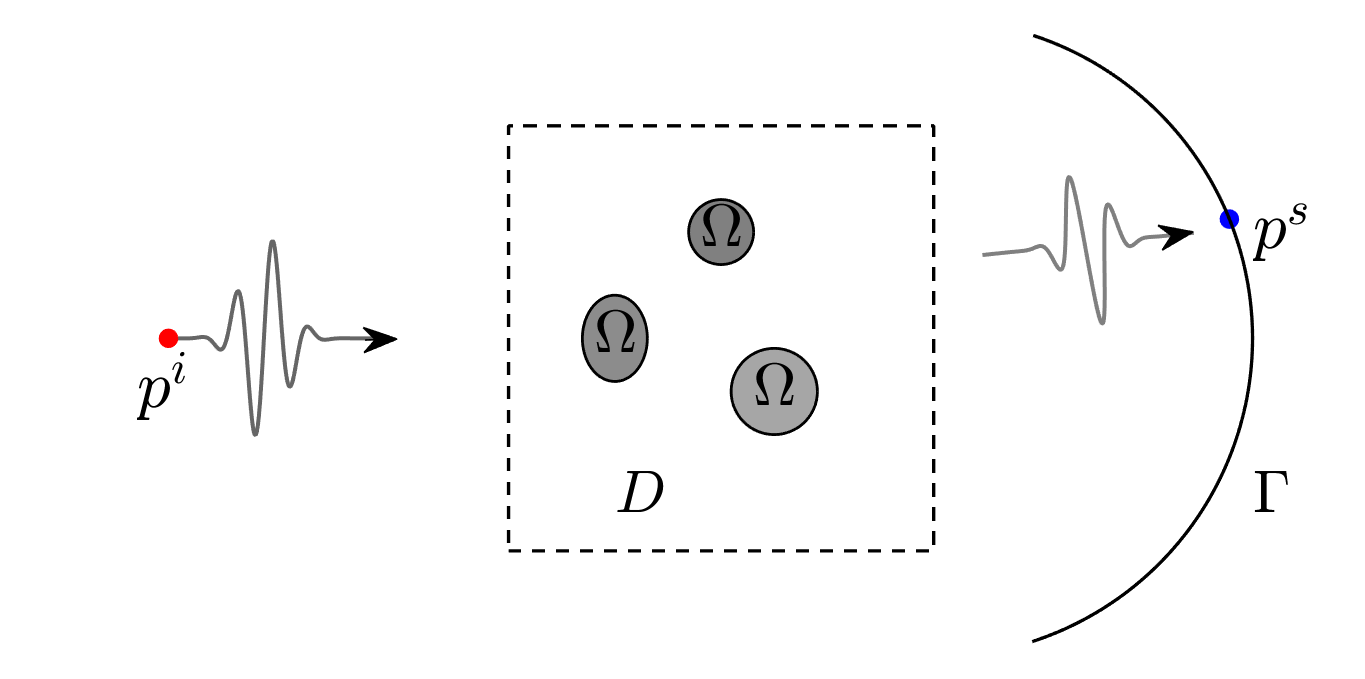}
    \caption{\label{fig-model} The schematic illustration of the time-domain scattering model. The red point denotes the location of an acoustic source that emits  causal temporal pulse wave. The union of domains labeled by $\Omega$ signifies medium scatterer.  The receivers are distributed on the curve $\Gamma$ to receive the scattered field $p^s$. The dashed square denotes the sampling domain $D$.}
\end{figure}

In recent decades, many numerical methods have been developed to deal with the above inverse scattering problem.  Among various reconstruction approaches, we are particularly interested in the so-called sampling-type methods. This class of approaches resort to establishing imaging functionals to distinguish whether a sampling point is inside or outside the scatterer's domain and hence can qualitatively identify the location and shape of the unknown scatterer\cite{Colton2019}.
 Significantly, sampling-type methods possess several advantages in contrast to conventional iterative approaches.
On the one hand, these sampling-type approaches are independent on the {\it a priori} information of scatterer's physical properties or boundary conditions.  On the other hand, these imaging functionals can be computed in a very efficient way and they are also robust to the noise.

Depending on whether directly using  the measured time-dependent data,  the sampling-type methods can be classified into two categories: time-domain ones and frequency-domain ones. The time-domain sampling methods  primarily utilize the measured time-dependent scattered field  to  determine the unknown objects.
We refer the reader to the time reversal method \cite{Ammari2013, Fink1997}, the synthetic aperture radar technique \cite{Cheney2001}, the total focusing method\cite{Holmes2005} and so on \cite{Klibanov2022, Sini2022}.
 However, if the measured data changes rapidly over a short period of time, it is difficult to use the time-domain methods \cite{Stankovic2014}.
In such cases, one usually transforms the  time-dependent data into frequency domain and  utilizes the frequency-domain  methods to recover unknown scatterers.
 For the frequency-domain sampling methods,
interested readers could refer the linear sampling method \cite{Cakoni2011, Li2010}, the factorization method\cite{Ammari2007, Kirsch2007}, the probe method \cite{Potthast2006}, the direct sampling method\cite{Ito2012, Li2013} and so on \cite{Chen2013, Liu2017}.

In comparison to time-domain sampling methods, the frequency-domain ones can offer better stability and computational efficiency, whereas they require a large amount of multi-static data on a large spatial aperture.
In order to decrease the number of sources and receivers, the linear sampling method and factorization method were extended to the time domain, see e.g. \cite{Chen2010, Guo2013, Guo2015, Cakoni2019}.
 Actually, the time-domain linear sampling method and factorization method were hindered by the theoretical difficulty associated with the interior transmission eigenvalue problems for a long time. Until recently, the solvability of the time-domain linear sampling method were rigorous justified in \cite{Prunty2019, Cakoni2021}.
Nevertheless, the classical direct sampling method \cite{Ito2012} does not suffer the  that interior transmission eigenvalue problems.
Moreover, to our best knowledge, the time-domain direct sampling method has not yet been established in the literature although there are some attempts
\cite{Guo2016, Wang}.
The major challenge of the time-domain direct sampling method lies in the design of the imaging functionals.
Hence, it has important theoretical and practical significance to  propose an efficient  time-domain direct sampling method.

The goal of this article is to develop a novel time-domain direct sampling method for reconstructing the locations and shapes of  unknown targets.
We first directly present the imaging/indicator functionals for the inverse acoustic scattering problems in both two and three dimensions as follows:
\begin{equation}\label{eq:indicator}
  \mathcal{I}^{(d)}(z)=\int_{-\infty}^{+\infty} \left| \int_{\Gamma} p^s(x,t+c_0^{-1}|x-z|)\,  \varphi_{\sigma}^{(d)}(x,t,z)  \, \mathrm{d}s(x)\right|^2 \mathrm{d}t, \quad d=2,\,3,
\end{equation}
where $z\in D$ denotes the sampling point, $p^s$ is the time-dependent scattered field  and  $\varphi_{\sigma}^{(d)}$ is the so-called test function, i.e.,
\begin{equation}\label{eq:test}
 \varphi_{\sigma}^{(d)}(x,t,z)=
  \begin{cases}
  \displaystyle \frac{\mathrm{e}^{-\sigma (t+ c_0^{-1}|x-z|)}}{\sqrt{8\pi c_0^{-1}|x-z|}}, & d=2, \bigskip\\
 \displaystyle \frac{\mathrm{e}^{-\sigma(t+  c_0^{-1}|x-z|)}}{4\pi|x-z|}, & d=3.
  \end{cases}
\end{equation}
Here $\sigma\geq 0$ is a  constant and it depends on the the waveform of the incident wave.



In this work, we first designed a novel imaging function \eqref{eq:indicator} for the time-domain direct sampling method. Then we establish the inherent connection between the time-domain and classical frequency-domain direct sampling method \cite{Ito2012} based on the Fourier-Laplace transform. Furthermore, we present rigorous mathematical analysis to justify the indicator behaviour of the proposed imaging functionals \eqref{eq:indicator} in the scenario involving the point-like scatterers. 
The promising features of our time-domain direct sampling method can be summarized in three aspects.
Firstly, the method involves formulating imaging functions by directly convoluting time-dependent scattered data with a test function.
Moreover, the imaging functions do not rely on any matrix inversion or forward solution process.
Thus, our approach is easy to implement with high computational efficiency.
Secondly, in comparison to the total focusing method\cite{Holmes2005}, the proposed method  not only utilizes the  travel time information of the recorded acoustic data but also makes use of the  echo amplitude of the scattered field.  Thus,  our approach theoretically has better imaging performance in recovering  unknown objects.    Thirdly,  the  proposed method has nearly no restrictions on the waveform of the incident wave. Regardless of whether the input signal is an impulse function, a periodic function or a tempered function, the imaging functions can demonstrate excellent performance when a suitable parameter $\sigma$ is selected. Hence, our approach can be generalized to more practical scenarios since it could handle a wide range of input signals.

The rest of the paper is organized as follows. Section 2 introduces the Sobolev spaces and the Fourier-Laplace transform. Moreover, we present some mathematical  preliminaries for the acoustic wave equation in both time- and frequency-domain.  In section 3, we first
establish the connection between the  time-domain  and frequency-domain direct sampling method. Further, we show the behaviour of the indicator functions.
Numerical experiments are conducted in section 4 to verify the promising features of our method.

\section{Fourier-Laplace transform and forward scattering problem}

 In this section, we provide some mathematical  preliminaries for the time-domain acoustic wave equation.

 We begin with the introduction of some space-time Sobolev spaces and the Fourier-Laplace transform.
 Let $X$ be a Hilbert space, we define the space $\mathcal{D}(\mathbb{R},X)$ as the $X$-valued $C_0^{\infty}$ functions on the real line with support in $(-\infty, \infty)$ and the space $\mathcal{D}'(\mathbb{R},X)$ as the associated $X$-valued distributions.
Moreover, we define the space $\mathcal{S}(\mathbb{R},X)$ as the Schwartz space of infinitely differentiable $X$-valued functions on the real
line and the space $\mathcal{S}'(\mathbb{R},X)$ as the corresponding tempered distributions. In order to  describe the Fourier-Laplace transform, we define the following Sobolev spaces
\begin{equation*}
\mathcal{L}_{\sigma}'(\mathbb{R},X):=\{ f\in \mathcal{D}'(\mathbb{R},X): \mathrm{e}^{-\sigma t}f \in \mathcal{S}'(\mathbb{R},X) \}, \quad \sigma\in \mathbb{R}.
\end{equation*}
 It is clear that the Fourier transform of $\mathrm{e}^{-\sigma t}f$ can be represented by
\begin{equation*}
\mathcal{F}[\mathrm{e}^{-\sigma t}f](\xi)=\int_{-\infty}^{+\infty} \mathrm{e}^{\mathrm{i}\xi t} \mathrm{e}^{-\sigma t} f(t) \, \mathrm{d}t, \quad \xi\in \mathbb{R}.
\end{equation*}
Let $\omega=\xi+\mathrm{i}\sigma$ for $\xi, \sigma\in \mathbb{R}$,  we define a half-plane by
\begin{equation*}
 \mathbb{C}_{\sigma_0}=\{\omega\in \mathbb{C}: \text{Im}(\omega)\geq \sigma_0>0\}
\end{equation*}
 Thus, the Fourier transform can be rewritten as the following Fourier-Laplace transform
\begin{equation}\label{eq:Fourier-transform}
  \widehat{f}(\omega)=\int_{-\infty}^{+\infty} \mathrm{e}^{\mathrm{i}\omega t} f(t) \, \mathrm{d}t, \quad \omega\in \mathbb{C}_{\sigma_0}.
\end{equation}
Indeed, the inverse Laplace transform is given by
\begin{equation}\label{eq:Laplace-transform}
  f(t)=\frac{1}{2\pi} \int_{-\infty+\mathrm{i}\sigma}^{+\infty+\mathrm{i}\sigma} \mathrm{e}^{-\mathrm{i}\omega t} \widehat{f}(\omega) \, \mathrm{d}\omega.
\end{equation}
 To characterize the time-dependent wave field, we also define the Hilbert space
\begin{equation*}
  H_{\sigma}^{m}(\mathbb{R},X):=\left\{f \in \mathcal{L}_{\sigma}'(\mathbb{R},X): \,
   \int_{-\infty+\mathrm{i}\sigma}^{+\infty+\mathrm{i}\sigma} |w|^{2m}\|\widehat{f}(\omega)\|_X^2 \, \mathrm{d}\omega<+\infty
  \right \}
\end{equation*}
for $m\in \mathbb{R}$ and $\sigma\in \mathbb{R}$. For more details, please refer to the  reference \cite{Sayas2016}. In additon, it is worthy to mention that the  Fourier-Laplace transform transfers to the  standard  Fourier transform when  $\Im \omega =0$.

 Next, we introduce the acoustic wave propagation and scattering in the inhomogeneous medium. Throughout, we assume that the incident field
 is generated by a single monopole source  with a causal temporal signal $\chi(t)\in \mathcal{C}^2(\mathbb{R})$, that is $\chi(t)\equiv 0$ for $t\leq 0$. Then the incident wave $p^i\in H_{\sigma}^{m}(\mathbb{R}, L^2(\mathbb{R}^d))$ emitted at the source location $y\in \mathbb{R}^d$ satisfies the homogeneous wave equations
\begin{equation}\label{eq:incient-wave}
\begin{aligned}
&   c_0^{-2}\partial_{tt}p^i(x,t)-\Delta p^i(x,t)=\chi(t)\, \delta(x-y), \quad (x,t)\in \mathbb{R}^d \times \mathbb{R},\\
&   p^i(x,0)=\partial_t p^i(x,0)=0,   \qquad \qquad \qquad \qquad  \  x\in \mathbb{R}^d,
   \end{aligned}
\end{equation}
where $\delta$ is the Dirac delta function. It is known that the fundamental solution to the wave equation is given by
\begin{equation*}
 \Phi(x,t;y):=
 \begin{cases}
 \displaystyle \frac{H(t-c_0^{-1}|x-y|)}{2\pi\sqrt{t^2-c_0^{-2}|x-y|^2}}, \quad \ d=2,\medskip\\
 \displaystyle \frac{\delta(t-c_0^{-1}|x-y|)}{4\pi|x-y|}, \quad\quad \quad d=3,
 \end{cases}
 x\neq y,
\end{equation*}
 where $H(t)$ denotes the Heaviside function.
In three dimensions, it is easy to verify that the solution to \eqref{eq:incient-wave} is
\begin{equation}\label{eq:source}
  p^i(x,t;y)=\frac{\chi\left(t-c_0^{-1}|x-y|\right)}{4\pi|x-y|}, \quad (x,t)\in \mathbb{R}^3 \backslash\{y\} \times \mathbb{R}.
\end{equation}
Combining \eqref{eq:main} and \eqref{eq:incient-wave}, one can find that the scattered field $p^s\in H_{\sigma}^{m}(\mathbb{R}, H^1(\mathbb{R}^d))$ admits the following equation
\begin{equation}\label{eq:scattered-wave}
\begin{aligned}
&  c_0^{-2}(x)\partial_{tt}p^s(x,t)-\Delta p^s(x,t)=\left(c_0^{-2}-c^{-2}(x)\right)\partial_{tt}p(x,t), \quad (x,t)\in \mathbb{R}^d\backslash \{y\} \times \mathbb{R},
\end{aligned}
\end{equation}
where $p:=p^i+p^s\in H_{\sigma}^{m}(\mathbb{R}, H^1(\mathbb{R}^d))$ denotes the total field.
Noting that $c(x)-c_0$ is compactly supported in $\Omega \in\mathbb{R}^d$, using the time-domain Green's formula,  the scattered field can be represented by
\begin{equation}\label{eq:tus}
  p^s(x,t)=\int_{\mathbb{R}} \int_{\Omega}
  \left(c_0^{-2}-c^{-2}(y)\right) \partial_{tt} p(y,\tau)  \, \Phi(x,t-\tau;y)\  \mathrm{d}y \,\mathrm{d}\tau.
\end{equation}

To facilitate the theoretical analysis, we first  convert the time domain to the frequency domain.
By applying the Fourier-Laplace transform to the scattered field and total field,  i.e.,
\begin{equation*}
  \widehat{p^s}(x,\omega):=\int_{-\infty}^{+\infty} \mathrm{e}^{\mathrm{i}\omega t} p^s(x,t) \, \mathrm{d}t, \quad
   \widehat{p}(x,\omega):=\int_{-\infty}^{+\infty} \mathrm{e}^{\mathrm{i}\omega t} p(x,t) \, \mathrm{d}t,
  \quad \omega\in \mathbb{C}_{\sigma_0},
\end{equation*}
equation \eqref{eq:scattered-wave} can be rewritten as
\begin{equation*}
  \frac{\omega^2}{c_0^2} \widehat{p^s}(x,\omega)+\Delta \widehat{p^s}(x,\omega) =\omega^2 \left(c_0^{-2}-c^{-2}(x)\right)\widehat{p}(x,\omega) , \quad  (x,\omega) \in \mathbb{R}^d\backslash\{ y\} \times \mathbb{C}_{\sigma_0}.
\end{equation*}
Recalling the Lippmann-Schwinger equation \cite[p.310]{Colton2019}, one can find that the solution to the last equation is given by
\begin{equation}\label{eq:us}
  \widehat{p^s}(x,\omega)=-\int_{\Omega}   \omega^2 \left(c_0^{-2}-c^{-2}(y)\right)\widehat{p}(y,\omega) \,  \Phi_{\omega}(x,y) \, \mathrm{d}y,
\end{equation}
where
 $\Phi_{\omega}(x,y)$ denotes the fundamental solution to the Helmholtz equation with frequency $\omega\in \mathbb{C}_{\sigma_0}$ \cite{s25}
\begin{equation}\label{eq:fundamental}
  \Phi_{\omega}(x,y)=
  \begin{cases}
  &\displaystyle\frac{\mathrm{i}}{4} H_0^{(1)}(\omega c_0^{-1}|x-y|), \quad d=2, \medskip\\
  &\displaystyle\frac{\mathrm{e}^{\mathrm{i}\omega c_0^{-1}|x-y|  }  }{4\pi|x-y|}, \qquad \qquad \ \ \ d=3,
  \end{cases}
 \ x\neq y.
\end{equation}
Here $H_0^{(1)}$ denotes the  Hankel function of the first kind with order zero. Thus, $  \widehat{p^s}$  is a frequecy-domain formula to the time-domain scattered field $p^s$ defined in \eqref{eq:tus}.

\section{Theoretical analysis of the  time-domain direct sampling method}
In this section, we aim to establish the rigorous theoretical analysis to determine the locations of the multiple small scatterers based on the
behavior of indicator functions \eqref{eq:indicator}. The key technique is to transform the time-domain imaging functions into the Laplace domain and discuss the properties of  imaging functions in the frequency domain.

We  first present the Laplace-domain representation of the indicator functions in the following theorem.
\begin{theorem}\label{thm:Fourier}
 Let the indicator/imaging functions of  time-domain direct sampling methods  be defined in \eqref{eq:indicator}.  By using the Laplace transform,   imaging functions can be  represented by
 \begin{subequations}\label{eq:indicator-fre}
\begin{align}
\label{eq:indicator-2D}&\mathcal{I}^{(2)}(z) = \frac{1}{2\pi} \int_{-\infty+\mathrm{i}\sigma}^{+\infty+\mathrm{i}\sigma}  \,
 \left| \int_{\Gamma} \widehat{p^s}(x,\omega) \frac{\mathrm{e}^{-\mathrm{i}\Re(\omega) c_0^{-1}|x-z|}}{\sqrt{8\pi  c_0^{-1}|x-z|}}\,  \mathrm{d}s(x)\right|^2 \,\mathrm{d}\omega, \medskip\\
\label{eq:indicator-3D}&\displaystyle \mathcal{I}^{(3)}(z)=\frac{1}{2\pi} \int_{-\infty+\mathrm{i}\sigma}^{+\infty+\mathrm{i}\sigma}
  \left| \int_{\Gamma} \widehat{p^s}(x,\omega)\frac{\mathrm{e}^{-\mathrm{i} \Re(\omega) c_0^{-1}|x-z|}}{4\pi|x-z|} \mathrm{d}s(x)\right|^2 \,\mathrm{d}\omega,
\end{align}
\end{subequations}
for two and three dimensions, respectively.  Here  $z\in D$ denotes the sampling point and   $\omega\in \mathbb{C}_{\sigma_0}$ is the frequency.
\end{theorem}
\begin{proof}
We first consider the three-dimensional case. In three dimensions, the indicator function \eqref{eq:indicator} is given by
\begin{equation*}
\mathcal{I}^{(3)}(z)=\int_{-\infty}^{+\infty} \left| \int_{\Gamma} p^s(x,t+c_0^{-1}|x-z|)  \frac{\mathrm{e}^{-\sigma(t+ c_0^{-1}|x-z|)}}{4\pi|x-z|}  \, \mathrm{d}s(x)\right|^2 \mathrm{d}t.
\end{equation*}
Due to $\overline{\exp(-\mathrm{i}\omega t)}=\exp(\mathrm{i}\omega t) \exp(-2\sigma t)$ for $\omega=\xi+\mathrm{i} \sigma\in \mathbb{C}_{\sigma_0}$, using the inverse Laplace transform \eqref{eq:Laplace-transform},  one can derive that
\begin{equation}\label{eq:I3-0}
\begin{aligned}
\mathcal{I}^{(3)}(z)
  &=\int_{-\infty}^{+\infty} \Bigg(\int_{\Gamma} \Bigg\{\frac{1}{2\pi} \int_{-\infty+\mathrm{i}\sigma}^{+\infty+\mathrm{i}\sigma}  \widehat{ p^s}(x,\omega) \mathrm{e}^{-\mathrm{i}\omega(t+c_0^{-1}|x-z|)} \,\mathrm{d}\omega \Bigg\} \frac{\mathrm{e}^{-\sigma c_0^{-1}|x-z|}}{4\pi|x-z|} \, \mathrm{d}s(x) \Bigg)\\
   &\quad \times \Bigg(\mathrm{e}^{-2\sigma t} \int_{\Gamma} p^s(x,t+c_0^{-1}|x-z|) \frac{\mathrm{e}^{-\sigma c_0^{-1}|x-z|}}{4\pi|x-z|} \, \mathrm{d}s(x)\Bigg) \ \mathrm{d}t\\
  &=\frac{1}{2\pi} \int_{-\infty+\mathrm{i}\sigma}^{+\infty+\mathrm{i}\sigma}
 \Bigg( \int_{\Gamma} \widehat{p^s}(x,\omega)\frac{\mathrm{e}^{-\mathrm{i}\xi c_0^{-1}|x-z|}}{4\pi|x-z|} \mathrm{d}s(x)\Bigg)\\
  &\quad \times
 \Bigg(\int_{\Gamma} \left\{\int_{-\infty}^{+\infty}  p^s(x,t+c_0^{-1}|x-z|) \mathrm{e}^{-\mathrm{i}\omega t}\, \mathrm{e}^{-2\sigma t}   \, \mathrm{d}t \right\} \frac{\mathrm{e}^{-\sigma c_0^{-1}|x-z|}}{4\pi|x-z|} \mathrm{d}s(x) \Bigg) \,\mathrm{d}\omega.
\end{aligned}
\end{equation}
Here and in what follows the overbar stands for the complex conjugate.
Furthermore, using the fact that $\overline{\exp(\mathrm{i}\omega t)}=\exp(-\mathrm{i}\overline{\omega} t)$ and the Fourier-Laplace transform \eqref{eq:Fourier-transform},  we can deduce that
\begin{equation*}
\begin{aligned}
&\int_{\Gamma} \left\{\int_{-\infty}^{+\infty}  p^s(x,t+c_0^{-1}|x-z|) \, \mathrm{e}^{-\mathrm{i}\omega t}\, \mathrm{e}^{-2\sigma t}   \, \mathrm{d}t \right\} \frac{\mathrm{e}^{-\sigma c_0^{-1}|x-z|}}{4\pi|x-z|} \,\mathrm{d}s(x)\\
&\quad =\int_{\Gamma}\overline{ \left( \int_{-\infty}^{+\infty}  p^s(x,t+c_0^{-1}|x-z|) \, \mathrm{e}^{\mathrm{i}\omega t}\, \mathrm{d}t \right ) \frac{\mathrm{e}^{-\sigma c_0^{-1}|x-z|}}{4\pi|x-z|}}\, \mathrm{d}s(x)\\
&\quad =\int_{\Gamma} \overline{\widehat{p^s}(x,\omega) \, \mathrm{e}^{-\mathrm{i}\omega c_0^{-1}|x-z|}\, \frac{\mathrm{e}^{-\sigma c_0^{-1}|x-z|}}{4\pi|x-z|}}\, \mathrm{d}s(x)\\
&\quad =\overline{\int_{\Gamma}\widehat{p^s}(x,\omega) \, \frac{\mathrm{e}^{-\mathrm{i}\Re(\omega)  c_0^{-1}|x-z|} }{4\pi|x-z|}\, \mathrm{d}s(x)}.
\end{aligned}
\end{equation*}
Substituting the last  equation into \eqref{eq:I3-0},  it follows that
\begin{equation*}
\begin{aligned}
\mathcal{I}^{(3)}(z)
&=\frac{1}{2\pi} \int_{-\infty+\mathrm{i}\sigma}^{+\infty+\mathrm{i}\sigma}
  \left| \int_{\Gamma} \widehat{p^s}(x,\omega)\frac{\mathrm{e}^{-\mathrm{i}\Re(\omega) c_0^{-1}|x-z|}}{4\pi|x-z|} \mathrm{d}s(x)\right|^2 \,\mathrm{d}\omega.
  \end{aligned}
\end{equation*}

Next, we prove the two-dimensional case.    In the 2D case, the indicator function is given by
\begin{equation*}
  \mathcal{I}^{(2)}(z)
  =\int_{-\infty}^{+\infty} \left| \int_{\Gamma} p^s(x,t+c_0^{-1}|x-z|)  \frac{\mathrm{e}^{-\sigma (t+c_0^{-1}|x-z|)}}{\sqrt{8\pi c_0^{-1}|x-z|}}  \, \mathrm{d}s(x)\right|^2 \mathrm{d}t.
\end{equation*}
Similarly,  by a straightforward calculation, we have
\begin{equation*}
\begin{aligned}
\mathcal{I}^{(2)}(z)
  =&\int_{-\infty}^{+\infty} \Bigg(\int_{\Gamma} \Bigg\{\frac{1}{2\pi} \int_{-\infty+\mathrm{i}\sigma}^{+\infty+\mathrm{i}\sigma}  \widehat{ p^s}(x,\omega) \mathrm{e}^{-\mathrm{i}\omega(t+c_0^{-1}|x-z|)} \,\mathrm{d}\omega \Bigg\} \frac{\mathrm{e}^{-\sigma c_0^{-1}|x-z|}}{\sqrt{8\pi c_0^{-1}|x-z|}} \, \mathrm{d}s(x) \Bigg)\\
   & \times \Bigg(\mathrm{e}^{-2\sigma t} \int_{\Gamma} p^s(x,t+c_0^{-1}|x-z|) \frac{\mathrm{e}^{-\sigma c_0^{-1}|x-z|}}{\sqrt{8\pi c_0^{-1}|x-z|}} \, \mathrm{d}s(x)\Bigg) \ \mathrm{d}t\\
  =&\ \frac{1}{2\pi} \int_{-\infty+\mathrm{i}\sigma}^{+\infty+\mathrm{i}\sigma}
 \Bigg( \int_{\Gamma} \widehat{p^s}(x,\omega)\frac{\mathrm{e}^{-\mathrm{i} \Re(\omega) c_0^{-1}|x-z|} }{\sqrt{8\pi  c_0^{-1}|x-z|}} \mathrm{d}s(x)\Bigg)\\
  &\ \times
 \Bigg(\overline{  \int_{\Gamma}\left\{\int_{-\infty}^{+\infty}  p^s(x,t+c_0^{-1}|x-z|) \mathrm{e}^{-\mathrm{i}\omega t}\, \mathrm{d}t \right\} \frac{\mathrm{e}^{-\sigma c_0^{-1}|x-z|}  }{\sqrt{8\pi c_0^{-1}|x-z|}} \, \mathrm{d}s(x)} \Bigg) \,\mathrm{d}\omega\\
&= \frac{1}{2\pi} \int_{-\infty+\mathrm{i}\sigma}^{+\infty+\mathrm{i}\sigma}  \,
 \left| \int_{\Gamma} \widehat{p^s}(x,\omega) \frac{\mathrm{e}^{-\mathrm{i}\Re(\omega) c_0^{-1}|x-z|}}{\sqrt{8\pi c_0^{-1}|x-z|}}\,  \mathrm{d}s(x)\right|^2 \,\mathrm{d}\omega.
\end{aligned}
\end{equation*}
This completes the proof.
\end{proof}

Noting that the asymptotic behavior of the Hankel functions is given by
\begin{equation}\label{eq:asymp-Hankel}
  H_0^{(1)}(r)=\sqrt{\frac{2}{\pi r}}\, \mathrm{e}^{\mathrm{i}\left(r-\frac{\pi}{4}\right)} \left\{ 1+\mathcal{O}\left(\frac{1}{r}\right)\right\}, \quad  r\rightarrow \infty,
\end{equation}
thus,  the fundamental solution in 2D can be  approximated by
\begin{equation*}
  \Phi_{\xi}(x,y)
   =\frac{\mathrm{i}}{4} H_0^{(1)}(\xi c_0^{-1}|x-y|)
\approx\frac{\mathrm{e}^{\mathrm{i}\xi c_0^{-1}|x-y|}\ \mathrm{e}^{\mathrm{i}\frac{\pi}{4}}  }{\sqrt{8\pi \xi c_0^{-1}|x-y|}}, \quad r \rightarrow \infty .
\end{equation*}
Substituting  the previous equation into \eqref{eq:indicator-2D},  we obtain
\begin{equation}\label{eq:indicator2D_new}
\begin{aligned}
\mathcal{I}^{(2)}(z)
   \approx &\frac{1}{2\pi} \int_{-\infty+\mathrm{i}\sigma}^{+\infty+\mathrm{i}\sigma} \xi   \,
 \left| \int_{\Gamma} \widehat{p^s}(x,\omega)   \overline{\Phi_{\xi }(x,z)} \, \mathrm{e}^{\mathrm{i}\frac{\pi}{4}}   \, \mathrm{d}s(x)\right|^2 \,\mathrm{d}\omega\\
  =&\frac{1}{2\pi} \int_{-\infty+\mathrm{i}\sigma}^{+\infty+\mathrm{i}\sigma}\xi    \,
 \left| \int_{\Gamma} \widehat{p^s}(x,\omega)   \overline{\Phi_{\xi }(x,z)}    \, \mathrm{d}s(x)\right|^2 \,\mathrm{d}\omega, \quad r \rightarrow \infty,
\end{aligned}
\end{equation}
where $\xi=\Re(\omega)$.
Correspondingly,  the formula \eqref{eq:indicator-3D} can be rewritten as
\begin{equation}\label{eq:indicator3D_new}
\mathcal{I}^{(3)}(z)
  =\frac{1}{2\pi} \int_{-\infty+\mathrm{i}\sigma}^{+\infty+\mathrm{i}\sigma}
 \left| \int_{\Gamma} \widehat{p^s}(x,\omega)   \overline{\Phi_{\xi }(x,z)  }  \, \mathrm{d}s(x)\right|^2 \,\mathrm{d}\omega.
\end{equation}

\begin{remark}
It is noted that the frequency-domain indicator functions \cite{Ito2012} are given by
\begin{equation*}
  \mathcal{J}^{(d)}(z)=\int_{\Gamma} \widehat{p^s}(x, \xi ) \overline{\Phi}_{\xi }(x,z)\,  \mathrm{d}s(x), \quad \xi\in \mathbb{R} \backslash \{0\}, \ d=2,3,
\end{equation*}
together with  formulas \eqref{eq:indicator2D_new} and \eqref{eq:indicator3D_new}, one can find that
\begin{equation*}
  \mathcal{I}^{(d)}(z)=\frac{1}{2\pi} \int_{-\infty}^{+\infty}\xi^{3-d}   \left|
  \mathcal{J}^{(d)}(z) \right|^2 \,\mathrm{d}\xi, \quad d=2,3,
\end{equation*}
when   $\Im(\omega)=0$.
 Thus,   our proposed time-domain direct sampling method  can be  viewed as the multi-frequency frequency-domain  direct sampling method when  $\sigma=0$.
Actually,  our proposed method is essentially a generalized version  of the multi-frequency frequency-domain  direct sampling method since the parameter $\sigma$ can also  be selected as $\sigma>0$.
\end{remark}


To fully characterize the time-domain direct sampling method, as stated in Theorem \ref{thm:Fourier}, we just need to analyze the properties of the indicator functions defined in \eqref{eq:indicator-fre}.
To accomplish this, we first present a crucial lemma.
 For simplification, we define
\begin{equation}\label{eq:G}
  G^{(d)}(z;y):=
  \begin{cases}
  &\displaystyle\int_{\Gamma} \Phi_{\omega}(x,y)  \frac{\mathrm{e}^{-\mathrm{i}\Re(\omega) c_0^{-1}|x-z|}}{\sqrt{8\pi c_0^{-1}|x-z|}}\, \mathrm{d}s(x), \quad d=2,\medskip\\
  &\displaystyle \int_{\Gamma} \Phi_{\omega}(x,y)  \frac{\mathrm{e}^{-\mathrm{i}\Re(\omega) c_0^{-1}|x-z|}}{4\pi c_0^{-1}|x-z|}\, \mathrm{d}s(x), \quad \ d=3,
  \end{cases}
\end{equation}
where $z\in \mathbb{R}^d$ denotes the sampling point and  $\Phi_{\omega}(x,y)$  is the fundamental solution defined in \eqref{eq:fundamental} with frequency $\omega\in \mathbb{C}_{\sigma_0}$.

\begin{lemma}\label{lemma1}
Suppose that $c_0>0$ is the background sound speed and $\sigma>0$ is a positive constant. Assume that  $\Gamma:=\{ x\in \mathbb{R}^d\backslash \overline{\Omega}: |x|=r \}$ is a circle ($d=2$) or a sphere ($d=3$) with radius $r$.  Let $G^{(d)}$ be defined in \eqref{eq:G},
for any $z\in \mathbb{R}^d$, it holds that
 \begin{equation*}
\begin{aligned}
 &\left|G^{(2)}(z;y) \right|\leq \frac{ \mathrm{e}^{-\sigma c_0^{-1}r}}{8\sqrt{\omega}c_0^{-1}}  I_0(\sigma c_0^{-1}|y|)  \left\{ 1+\mathcal{O}\left(\frac{1}{r}\right)  \right\}, \\
 & \left| G^{(3)}(z;y) \right|
  \leq  \frac{\mathrm{e}^{-\sigma c_0^{-1}r} }{4\pi c_0^{-2}}  i_0(\sigma c_0^{-1}|y|) \left\{ 1+\mathcal{O}\left(\frac{1}{r}\right)  \right\},
  \end{aligned}
\quad r\rightarrow \infty,
    \end{equation*}
 where $I_0$ and $i_0$ are respectively the modified Bessel function and modified spherical Bessel functions of the first kind with order zero. In paricular, the equalities hold if and only if $z=y$.

Moreover,  the integrals defined in \eqref{eq:G} have the following asymptotic behaviours
 \begin{equation*}
 \begin{aligned}
 &\left|G^{(2)}(z;y) \right|= \frac{ \mathrm{e}^{-\sigma c_0^{-1}r}}{8\pi\sqrt{\omega}c_0^{-1}} \mathcal{O}\left(\frac{1}{\sqrt{\Re(\omega)|z-y|}}\right)  \left\{ 1+\mathcal{O}\left(\frac{1}{r}\right)  \right\},\\
  &\left|G^{(3)}(z;y) \right|
   = \frac{\mathrm{e}^{-\sigma c_0^{-1}r} }{ 16\pi^2 c_0^{-2}}  \mathcal{O}\left(\frac{1}{ \Re(\omega)|z-y|}\right)  \left\{ 1+\mathcal{O}\left(\frac{1}{r}\right)  \right\},
    \end{aligned}
    \end{equation*}
as   $r \rightarrow \infty$ and $\Re(\omega)|z-y| \rightarrow \infty$.

\end{lemma}
\begin{proof}
We first consider the two-dimensional case.  According to \eqref{eq:asymp-Hankel},  the fundamental solution has the following asymptotic behaviour
\begin{equation*}
  \Phi_{\omega}(x,y)=\frac{\mathrm{i}}{4} H_0^{(1)}(\omega c_0^{-1}|x-y|)=\frac{\mathrm{e}^{\mathrm{i}\omega c_0^{-1}|x-y|}\ \mathrm{e}^{\mathrm{i}\frac{\pi}{4}}  }{\sqrt{8\pi \omega c_0^{-1}|x-y|}}\left\{ 1+\mathcal{O}\left(\frac{1}{|x|}\right)  \right\}.
\end{equation*}
Using the asymptotic formula
\begin{equation*}
  |x-y|=|x|-\hat{x}\cdot y+\mathcal{O}\left( \frac{1}{|x|}\right),
\end{equation*}
we have
\begin{equation}\label{eq:int-2D}
\begin{aligned}
 G^{(2)}(z;y)&= \int_{\Gamma} \Phi_{\omega}(x,y)  \frac{\mathrm{e}^{-\mathrm{i}\Re(\omega) c_0^{-1}|x-z|}}{\sqrt{8\pi c_0^{-1}|x-z|}}\, \mathrm{d}s(x)\\
  &=\frac{\mathrm{e}^{\mathrm{i}\frac{\pi}{4}} }{8\pi \sqrt{\omega}c_0^{-1}}\int_{\Gamma} \frac{\mathrm{e}^{\mathrm{i}\omega c_0^{-1}|x-y|}\  }{\sqrt{|x-y|}}  \frac{\mathrm{e}^{-\mathrm{i}\Re(\omega) c_0^{-1}|x-z|}}{\sqrt{|x-z|}}\, \left\{ 1+\mathcal{O}\left(\frac{1}{|x|}\right)  \right\} \mathrm{d}s(x)\\
  &=\frac{\mathrm{e}^{\mathrm{i}\frac{\pi}{4}} }{8\pi \sqrt{\omega}c_0^{-1}}\int_{\Gamma} \frac{\mathrm{e}^{-\sigma c_0^{-1}(|x|-\hat x\cdot y)}\  }{\sqrt{|x|}}  \frac{\mathrm{e}^{\mathrm{i}\Re(\omega) c_0^{-1}\hat{x}\cdot(z-y)}}{\sqrt{|x|}}\,  \left\{ 1+\mathcal{O}\left(\frac{1}{|x|}\right)  \right\} \mathrm{d}s(x)\\
    &=\frac{\mathrm{e}^{\mathrm{i}\frac{\pi}{4}} \mathrm{e}^{-\sigma c_0^{-1}r}}{8\pi \sqrt{\omega}c_0^{-1}}\int_{\mathbb{S}^1}  \mathrm{e}^{\mathrm{i}\Re(\omega) c_0^{-1}\hat{x}\cdot(z-y)}\  \mathrm{e}^{\sigma c_0^{-1}\hat x\cdot y}\, \mathrm{d}s(\hat{x})  \left\{ 1+\mathcal{O}\left(\frac{1}{r}\right)  \right\}.
  \end{aligned}
\end{equation}
Here and in what follows, $\mathbb{S}^{d-1}:=\{x\in \mathbb{R}^d: |x|=1\}$ denotes a unit circle/sphere in $\mathbb{R}^d$, $d=2,3$.
In two dimensions,  the generating function is given by
\begin{equation*}
  \mathrm{e}^{\sigma c_0^{-1}\hat x\cdot y}=I_0(\sigma c_0^{-1}|y|)+2\sum_{n=1}^{\infty} I_n(\sigma c_0^{-1}|y|)\cos(n \theta),
\end{equation*}
where  $I_n$ is the modified Bessel function of the first kind with order $n$ and $\theta$ denotes the  angle between $\hat x$ and $y$. From the last equation, one can derive that
\begin{equation*}
  \begin{aligned}
  &\left|\int_{\mathbb{S}^1}  \mathrm{e}^{\mathrm{i}\Re(\omega) c_0^{-1}\hat{x}\cdot(z-y)}\  \mathrm{e}^{\sigma c_0^{-1}\hat x\cdot y}\, \mathrm{d}s(\hat{x})\right| \\
  &\leq\int_{\mathbb{S}^1}  \left|\mathrm{e}^{\mathrm{i}\Re(\omega) c_0^{-1}\hat{x}\cdot(z-y)}\right|\cdot  \left| \mathrm{e}^{\sigma c_0^{-1}\hat x\cdot y}\right|\, \mathrm{d}s(\hat{x})\\
  &=\int_0^{\pi}  \left( I_0(\sigma c_0^{-1}|y|)+2\sum_{n=1}^{\infty} I_n(\sigma c_0^{-1}|y|)\cos(n \theta)\right) \, \mathrm{d}\theta \\
  &=\pi I_0(\sigma c_0^{-1}|y|),
  \end{aligned}
\end{equation*}
where the equality holds if and only if $z=y$.  Thus, for any $z\in \mathbb{R}^2$,  we have
\begin{equation*}
 \left| G^{(2)}(z;y)\right|=\left| \int_{\Gamma} \Phi_{\omega}(x,y)  \frac{\mathrm{e}^{-\mathrm{i}\Re(\omega) c_0^{-1}|x-z|}}{\sqrt{8\pi c_0^{-1}|x-z|}}\, \mathrm{d}s(x)\right|\leq \frac{ \mathrm{e}^{-\sigma c_0^{-1}r}}{8\sqrt{|\omega|}c_0^{-1}}  I_0(\sigma c_0^{-1}|y|)  \left\{ 1+\mathcal{O}\left(\frac{1}{r}\right)  \right\},
\end{equation*}
where the equality holds if and only if $z=y$.

Next,  we consider the three-dimensional case. By a straightforward calculation,  we have
\begin{equation}\label{eq:int-3D}
\begin{aligned}
  G^{(3)}(z;y)&=\int_{\Gamma} \Phi_{\omega}(x,y)  \frac{\mathrm{e}^{-\mathrm{i}\Re(\omega) c_0^{-1}|x-z|}}{4\pi c_0^{-1}|x-z|}\, \mathrm{d}s(x)\\
  &=\frac{1 }{(4\pi c_0^{-1})^2}\int_{\Gamma} \frac{\mathrm{e}^{\mathrm{i}\omega c_0^{-1}|x-y|}\  }{|x-y|}  \frac{\mathrm{e}^{-\mathrm{i}\Re(\omega) c_0^{-1}|x-z|}}{|x-z|}\,  \mathrm{d}s(x)\\
    &=\frac{\mathrm{e}^{-\sigma c_0^{-1}r} }{(4\pi c_0^{-1})^2}\int_{\mathbb{S}^2}  \mathrm{e}^{\mathrm{i}\Re(\omega) c_0^{-1}\hat{x}\cdot(z-y)}\  \mathrm{e}^{\sigma c_0^{-1}\hat x\cdot y}\,  \mathrm{d}s(\hat{x})  \left\{ 1+\mathcal{O}\left(\frac{1}{r}\right)  \right\}.
  \end{aligned}
    \end{equation}
In three dimensions,    the generating function  is given by
\begin{equation*}
  \mathrm{e}^{\sigma c_0^{-1}\hat x\cdot y}=4\pi\sum_{n=0}^{\infty}\sum_{m=-n}^n i_n(\sigma c_0^{-1}|y|) Y_n^m\left(\frac{y}{|y|}\right) \overline{Y_n^m}(\hat x),
\end{equation*}
where $Y_n^m$ denote the spherical harmonics.  Due to  $Y_0^0(\hat x)=1/\sqrt{4\pi}$, one has
\begin{equation*}\label{eq:Y0}
  \int_{\mathbb{S}^2} \overline{Y_n^m}(\hat x) \, \mathrm{d}s(\hat x)= \sqrt{4\pi}\int_{\mathbb{S}^2}  Y_0^0(\hat x)  \overline{Y_n^m}(\hat x)\, \mathrm{d}s(\hat x)=
  \begin{cases}
  &  \sqrt{4\pi}, \quad  n=0,\\
  & 0, \qquad  \ \, n\geq 1.
  \end{cases}
\end{equation*}
From the previous two equations, one has
\begin{equation*}
  \begin{aligned}
  &\left|\int_{\mathbb{S}^2}  \mathrm{e}^{\mathrm{i}\Re(\omega) c_0^{-1}\hat{x}\cdot(z-y)}\  \mathrm{e}^{\sigma c_0^{-1}\hat x\cdot y}\, \mathrm{d}s(\hat{x})\right| \\
  &\leq\int_{\mathbb{S}^2}  \left|\mathrm{e}^{\mathrm{i}\Re(\omega) c_0^{-1}\hat{x}\cdot(z-y)}\right|\cdot  \left| \mathrm{e}^{\sigma c_0^{-1}\hat x\cdot y}\right|\, \mathrm{d}s(\hat{x})\\
  &=\int_{\mathbb{S}^2}  \left( 4\pi\sum_{n=0}^{\infty}\sum_{m=-n}^n i_n(\sigma c_0^{-1}|y|) Y_n^m\left(\frac{y}{|y|}\right) \overline{Y_n^m}(\hat x)\right) \, \mathrm{d}s(\hat{x})\\
  &=4\pi i_0(\sigma c_0^{-1}|y|),
  \end{aligned}
\end{equation*}
where $i_n$ are the modified  spherical Bessel functions of the first kind with order $n$. Moreover,  the equality holds if and only if $z=y$.
Therefore,  for any $z\in \mathbb{R}^3$, we obtain
\begin{equation*}
 \left| G^{(3)}(z;y)\right|= \left|\int_{\Gamma} \Phi_{\omega}(x,y)  \frac{\mathrm{e}^{-\mathrm{i}\Re(\omega) c_0^{-1}|x-z|}}{4\pi c_0^{-1}|x-z|}\, \mathrm{d}s(x)\right|
  \leq   \frac{\mathrm{e}^{-\sigma c_0^{-1}r} }{4\pi c_0^{-2}}  i_0(\sigma c_0^{-1}|y|) \left\{ 1+\mathcal{O}\left(\frac{1}{r}\right)  \right\},
    \end{equation*}
where the equality holds if and only if $z=y$.

Finally, we show the  asymptotic behaviour of the integrals $G^{(d)}(z;y)$. Recalling the asymptotic behaviour of oscillatory integrals of the first kind in \cite[proposition 6, p.344]{Stein93},
for any fixed point $y$,   one can find that
\begin{equation*}
\left|\int_{\mathbb{S}^{d-1}}  \mathrm{e}^{\mathrm{i}\Re(\omega) c_0^{-1}\hat{x}\cdot(z-y)}\  \mathrm{e}^{\sigma c_0^{-1}\hat x\cdot y}\, \mathrm{d}s(\hat{x}) \right| =\mathcal{O}\left( (\Re(\omega)|z-y|)^{-\frac{d-1}{2}}\right),\quad  \Re(\omega)|z-y|\rightarrow \infty,
\end{equation*}
where $d=2,3$ denotes two- or three-dimensional case.  Hence, substituting the previous equation into \eqref{eq:int-2D} and \eqref{eq:int-3D} respectively,  one  can obtain
 \begin{equation*}
 \begin{aligned}
 &\left| G^{(2)}(z;y)\right|= \frac{ \mathrm{e}^{-\sigma c_0^{-1}r}}{8\pi\sqrt{ \omega}c_0^{-1}} \mathcal{O}\left(\frac{1}{\sqrt{\Re(\omega)|z-y|}}\right)  \left\{ 1+\mathcal{O}\left(\frac{1}{r}\right)  \right\}, \\
  &\left|G^{(3)}(z;y)\right|
   = \frac{\mathrm{e}^{-\sigma c_0^{-1}r} }{ 16\pi^2 c_0^{-2}}  \mathcal{O}\left(\frac{1}{\Re(\omega)|z-y|}\right)  \left\{ 1+\mathcal{O}\left(\frac{1}{r}\right)  \right\},
    \end{aligned}
    \end{equation*}
as $r \rightarrow \infty$ and $\Re(\omega)|z-y|\rightarrow \infty$.

The proof is completed.
\end{proof}

Lemma \ref{lemma1} shows that the integrals $|G^{(d)}(z;y)|$ attain  maximum values if and only if $z=y$. Moreover, the integrals $|G^{(d)}(z;y)|$ rapidly decay and go to zero when $z$ is away from $y$. In what follows, we present a special case to show the properties of the integrals $G^{(3)}(z;y)$.
Using the Jacobi-Anger expansion
\begin{equation*}
  \mathrm{e}^{\mathrm{i}z \cos \theta}= \sum_{n=0}^{\infty} \mathrm{i}^n (2n+1) j_n(z) P_n(\cos \theta),
\end{equation*}
and the properties of the Legendre polynomials
\begin{equation*}
  \int_{0}^{\pi} P_n(\cos \theta) \sin \theta \, \mathrm{d}\theta=
  \begin{cases}
  & 2, \quad n=0, \medskip \\
  & 0, \quad n\geq 1,
  \end{cases}
\end{equation*}
if $(z-y)=ky,\ k\geq 0$, then one has
\begin{equation*}
\begin{aligned}
&\int_{\mathbb{S}^{2}}  \mathrm{e}^{\mathrm{i}\Re(\omega) c_0^{-1}\hat{x}\cdot(z-y)}\  \mathrm{e}^{\sigma c_0^{-1}\hat x\cdot y}\, \mathrm{d}s(\hat{x})\\
&=\int_{0}^{2\pi} \int_{0}^{\pi}  \mathrm{e}^{\mathrm{i}(\Re(\omega) c_0^{-1}|z-y|-\mathrm{i}\sigma c_0^{-1} |y|)\cos \theta} \sin \theta\,  \mathrm{d} \theta \mathrm{d} \varphi\\
&=\int_{0}^{2\pi} \int_{0}^{\pi} \sum_{n=0}^{\infty} \mathrm{i}^n (2n+1) j_n(\Re(\omega) c_0^{-1}|z-y|-\mathrm{i}\sigma c_0^{-1} |y|) P_n(\cos \theta) \sin \theta\,  \mathrm{d} \theta \mathrm{d} \varphi\\
&=4\pi j_0(\Re(\omega) c_0^{-1}|z-y|-\mathrm{i}\sigma c_0^{-1} |y|).
\end{aligned}
\end{equation*}
Therefore,  we have
\begin{equation}\label{eq:G3}
   G^{(3)}(z;y)=\frac{\mathrm{e}^{-\sigma c_0^{-1}r} }{4\pi c_0^{-2}}  j_0(\Re(\omega) c_0^{-1}|z-y|-\mathrm{i}\sigma c_0^{-1} |y|) \left\{ 1+\mathcal{O}\left(\frac{1}{r}\right)  \right\}, \ r\rightarrow \infty.
\end{equation}
On the one hand, if $z=y$, in terms of
$i_n(x)= \mathrm{i}^{-n} j_n(\mathrm{i}x)$ and $i_n(-x)=(-1)^n i_n(x)$ \cite{Abramowitz1965},
one can find that
\begin{equation*}
\begin{aligned}
   G^{(3)}(y;y)
   &=\frac{\mathrm{e}^{-\sigma c_0^{-1}r} }{4\pi c_0^{-2}}  j_0(-\mathrm{i}\sigma c_0^{-1} |y|) \left\{ 1+\mathcal{O}\left(\frac{1}{r}\right)  \right\}\\
   &=\frac{\mathrm{e}^{-\sigma c_0^{-1}r} }{4\pi c_0^{-2}}  i_0(\sigma c_0^{-1} |y|) \left\{ 1+\mathcal{O}\left(\frac{1}{r}\right)  \right\}, \quad r\rightarrow \infty.
   \end{aligned}
\end{equation*}
On the other hand, due to $j_0(z)=\sin z/z$, together with \eqref{eq:G3}, one can verify that
\begin{equation*}
    G^{(3)}(z;y)=\frac{\mathrm{e}^{-\sigma c_0^{-1}r} }{4\pi c_0^{-2}}  \mathcal{O}\left(\frac{1}{\Re(\omega)|z-y|}\right) \left\{ 1+\mathcal{O}\left(\frac{1}{r}\right)  \right\}, \quad\Re(\omega)|z-y| \rightarrow \infty.
\end{equation*}

We are now ready to present the behavior of the imaging functions \eqref{eq:indicator}, which plays an important role in determining the locations of the point-like scatterers.
Before the discussion, we give two assumptions that play important roles in the sequel. Firstly, we assume that the medium domain $\Omega$ consists of a finite number of well-separated small domains, i.e.,
\begin{equation}\label{eq:assum1}
\Omega:=\bigcup_{j=1}^N \Omega_j, \quad N\in \mathbb{N}.
\end{equation}
Secondly,  let $\Re(\omega_0):=\inf_{\omega\in \mathbb{C}_{\sigma_0}} \Re(\omega)$ be  the lowest frequency of the  incident signal $\chi(t)$ such that
\begin{equation}\label{eq:assum3}
\displaystyle \Re(\omega_0) \gg \frac{2c_0}{\min \limits_{1\leq j, j'\leq N \atop j\neq j'} \text{dist}(\Omega_j, \Omega_{j'}) },
\end{equation}
where $c_0$ is the background sound speed. In the following, we denote by $B(y_0,R)$ the circle/ball centered at $y_0$ with $R$, namely, $B(y_0,R)=\{y\in \mathbb{R}^d: |y-y_0|<R \}$. 

\begin{theorem}\label{thm:main}
Let $\Gamma:=\{ x\in \mathbb{R}^d\backslash \overline{\Omega}: |x|=r \}$ denote  a circle ($d=2$) or a sphere ($d=3$) with radius $r$,  $c_0>0$ denotes the background sound speed and $\sigma>0$ be  a positive constant.
Suppose that the assumptions \eqref{eq:assum1} and \eqref{eq:assum3}
hold, then the indicator functions $\mathcal{I}^{(d)}(z)$ described in \eqref{eq:indicator} have the following asymptotic behaviour
\begin{align}
&\label{eq:Indicator-approx}\!\!\! \mathcal{I}^{(d)}(z) \!\leq \!\mathrm{e}^{-2\sigma c_0^{-1}r} \left( M_j+\mathcal{O}\!\left(r^{-1} \right)+
\mathcal{O}\left( (\Re(\omega_0)L)^{-\frac{d-1}{2}} \right) +  \mathcal{O}(\epsilon) \right), \ z\in B(y_j, L/2),\medskip\\
&\mathcal{I}^{(d)}(z)=\mathrm{e}^{-2\sigma c_0^{-1}r} \, \mathcal{O}\left( \frac{1}{(\Re(\omega_0)L)^{d-1}} \right), \  z\in \Omega\backslash \cup_{j=1}^N B(y_j, L/2),
\end{align}
where $\mathcal{O}(\epsilon)$ denotes the discretization error and  $L:=\min\limits_{1\leq j, j'\leq N \atop j\neq j'} \rm{dist}(\Omega_j, \Omega_{j'}) $
 and $M_j$ is a positive constant, i.e.,
\begin{equation*}
M_j:=
\begin{cases}
&\displaystyle \frac{c_0^2}{128\pi}h_j^2 I_0^2(\sigma c_0^{-1} y_j) \int_{-\infty+\mathrm{i}\sigma}^{+\infty+\mathrm{i}\sigma} \left|\omega^{\frac{3}{2}} \widehat{p}(y_j,\omega)\right|^2 \,\mathrm{d}\omega, \quad d=2, \medskip \\
&\displaystyle \frac{c_0^4}{32\pi^3}h_j^2 i_0^2(\sigma c_0^{-1} y_j) \int_{-\infty+\mathrm{i}\sigma}^{+\infty+\mathrm{i}\sigma} \left|\omega^2 \widehat{p}(y_j,\omega)\right|^2 \,\mathrm{d}\omega, \quad \ \, d=3.
\end{cases}
\end{equation*}
Here the point  $y_j$ is in the $j$-th element $\Omega_j$ and $h_j:=(c^{-2}(y_j)-c_0^{-2})|\Omega_j|$ denotes the weight with $|\Omega_j|$ being the measure of the element $\Omega_j$.
In particular, the equality in
\eqref{eq:Indicator-approx} holds if and only if $z=y_j$, $j=1,2, ... , N$.
\end{theorem}

\begin{proof}
Without loss of generality, we just consider the three dimensional case.
By using the retangular quadrature rule and  formula \eqref{eq:us}, one can deduce that
\begin{equation}\label{eq:approx}
\begin{aligned}
 &\int_{\Gamma} \widehat{p^s}(x,\omega)  \frac{\mathrm{e}^{-\mathrm{i} \Re(\omega) c_0^{-1}|x-z|}}{4\pi|x-z|} \, \mathrm{d}s(x)\\
 &\qquad = -\int_{\Omega}\left(\omega^2 \left(c_0^{-2}-c^{-2}(y)\right) \widehat{p}(y,\omega)\int_{\Gamma} \Phi_{\omega}(x,y) \frac{\mathrm{e}^{-\mathrm{i} \Re(\omega) c_0^{-1}|x-z|}}{4\pi|x-z|}\, \mathrm{d}s(x) \right)  \mathrm{d}y\\
   &\qquad= \sum_{j=1}^N \left( h_j \omega^2 \widehat{p}(y_j,\omega) \int_{\Gamma} \Phi_{\omega}(x,y_j) \frac{\mathrm{e}^{-\mathrm{i} \Re(\omega) c_0^{-1}|x-z|}}{4\pi|x-z|}\, \mathrm{d}s(x)\right)\Big\{1+\mathcal{O}(\epsilon)\Big\},
  \end{aligned}
\end{equation}
where $h_j:=(c^{-2}(y_j)-c_0^{-2})|\Omega_j|$ denotes the weight and $\mathcal{O}(\epsilon)$ denotes the  discretization error.

If $z\in B(y_j, L/2)$, using the equation \eqref{eq:approx} and lemma \ref{lemma1}, one can derive that
\begin{equation*}
\begin{aligned}
 \mathcal{I}^{(3)}(z)
 &=\frac{1}{2\pi} \int_{-\infty+\mathrm{i}\sigma}^{+\infty+\mathrm{i}\sigma}
  \left| \int_{\Gamma} \widehat{p^s}(x,\omega)  \frac{\mathrm{e}^{-\mathrm{i} \Re(\omega) c_0^{-1}|x-z|}}{4\pi|x-z|} \mathrm{d}s(x)\right|^2 \,\mathrm{d}\omega\\
 &=\frac{1}{2\pi} \int_{-\infty+\mathrm{i}\sigma}^{+\infty+\mathrm{i}\sigma}
\Bigg| h_j \omega^2 \widehat{p}(y_j,\omega) \int_{\Gamma} \Phi_{\omega}(x,y_j) \frac{\mathrm{e}^{-\mathrm{i} \Re(\omega) c_0^{-1}|x-z|}}{4\pi|x-z|}\, \mathrm{d}s(x)\\
&\quad+ \sum_{\substack{j'=1\\
j'\neq j}}^N \left( h_{j'} \omega^2 \widehat{p}(y_{j'},\omega) \int_{\Gamma} \Phi_{\omega}(x,y_{j'}) \frac{\mathrm{e}^{-\mathrm{i} \Re(\omega) c_0^{-1}|x-z|}}{4\pi|x-z|}\, \mathrm{d}s(x) \right)\Bigg|^2 \mathrm{d}\omega \cdot \Big\{1+\mathcal{O}(\epsilon)\Big\}\\
&\leq\frac{1}{2\pi}\int_{-\infty+\mathrm{i}\sigma}^{+\infty+\mathrm{i}\sigma}  \Bigg|h_j \omega^2 \widehat{p}(y_j,\omega) \frac{\mathrm{e}^{-\sigma c_0^{-1}r}}{4\pi c_0^{-2}} i_0(\sigma c_0^{-1} y_j)  \left\{1+\mathcal{O}\left( \frac{1}{r}\right)\right\}   \\
& \quad + \sum_{\substack{j'=1\\
j'\neq j}}^N \left( h_{j'} \omega^2 \widehat{p}(y_{j'},\omega) \frac{\mathrm{e}^{-\sigma c_0^{-1}r} }{ 16\pi^2 c_0^{-2}}  \mathcal{O}\left(\frac{1}{ \Re(\omega)|z-y|}\right)  \left\{ 1+\mathcal{O}\left(\frac{1}{r}\right)  \right\}\right)\Bigg|^2 \mathrm{d}\omega \cdot \Big\{1+\mathcal{O}(\epsilon)\Big\}.
 \end{aligned}
\end{equation*}
Let
\begin{equation*}
  M_j:=\frac{1}{2\pi}\int_{-\infty+\mathrm{i}\sigma}^{+\infty+\mathrm{i}\sigma}  \Bigg|h_j \omega^2 \widehat{p}(y_j,\omega) \frac{1}{4\pi c_0^{-2}} i_0(\sigma c_0^{-1} y_j) \Bigg|^2 \mathrm{d}\omega,
\end{equation*}
due to $p(x,t)\in H_{\sigma}^{m}(\mathbb{R}, H^1(\mathbb{R}^d))$ and $y_j$ is a fixed point,  one can find that $M_j$  is a positive  constant and it is given by
\begin{equation*}
M_j= \frac{c_0^4}{32\pi^3}h_j^2 i_0^2(\sigma c_0^{-1} y_j) \int_{-\infty+\mathrm{i}\sigma}^{+\infty+\mathrm{i}\sigma} \left|\omega^2 \widehat{p}(y_j,\omega)\right|^2 \,\mathrm{d}\omega.
\end{equation*}
Furthermore, by a straightforward calculation,  we obtain
\begin{equation*}
   \mathcal{I}^{(3)}(z)\leq \mathrm{e}^{-2\sigma c_0^{-1}r}   \left\{M_j +\mathcal{O}\left( \frac{1}{r}\right)+
\mathcal{O}\left( \frac{1}{\Re(\omega_0)L}\right)+  \mathcal{O}(\epsilon) \right\}, \quad  z\in B(y_j, L/2).
\end{equation*}
By  lemma  \ref{lemma1},  it is clear to see that the above equality holds if and only if  $z=y_j$.

Correspondingly, if $z\in \Omega\backslash \cup_{j=1}^N B(y_j, L/2)$, from \eqref{eq:approx} and lemma \ref{lemma1}, and using a similar calculation, one can get
\begin{equation*}
\mathcal{I}^{(3)}(z)=\mathrm{e}^{-2\sigma c_0^{-1}r}  \mathcal{O}\left( \frac{1}{(\Re(\omega_0)L)^2}\right).
  \end{equation*}
This completes the proof.

\end{proof}

\begin{remark}
Theorem \ref{thm:main} asserts that the imaging functional \eqref{eq:indicator} attains its local maximum at the  location of the targets and  it decays rapidly when the sampling point is away from the scatterers. Thus, one can identify the point-like scatterers based on the behaviour of the imaging functionals. Furthermore, to improve the resolution of the imaging,  one should choose a higher lowest frequency $\Re(\omega_0)$, a larger observation radius $r$ and finer sampling grid.
\end{remark}

\begin{remark}\label{rem:sigma}
It is worthy to emphasize that the parameter $\sigma$  plays a significant role in the imaging functionals. On the one hand, if the $L^2$ norm of the incident wave is unbounded,  a sufficiently large value of $\sigma$ must be chosen to ensure that the imaging functional \eqref{eq:indicator} remains bounded for any sampling points located away from the measured surface.  In particular, we can set $\sigma=0$ when the incident wave is a Gaussian-modulated  pulse wave.
On the other hand, there usually exists some measured noise in the measurement data. Therefore,  we must select a small enough value $\sigma$ to ensure that the maximum value of  the imaging functional is not annihilated by  the measured noise.  Hence, it is important to select  an appropriate $\sigma$ for  the inversion and imaging. For a specific model, the optimal choice of such $\sigma$ is still open.

\end{remark}

\section{Numerical experiments}

 In this section, several two dimensional numerical examples are presented to illustrate the feasibility and the effectiveness of the proposed time-domain direct sampling methods.

To generate synthetic scattered field data, we use the finite element method to solve the forward problem  of \eqref{eq:main} and \eqref{eq:source} in the time domain. Here, the quadratic finite element discretization in space and Crank-Nicolson scheme in time are employed, and the unbounded exterior domain is truncated by an absorbing boundary condition. The mesh of the forward solver is gradually refined until the relative error of the measured data is below 0.1\%. To test the stability of the proposed method, random noise was added to the synthetic scattered data $p^s(x, t)$. The noisy data are given by the following formula
\begin{equation*}
  p_{\delta}^s:=p^s(1+\delta R),
\end{equation*}
where $R$ are point-wise uniform random numbers, ranging from $-1$ to $1$, and $\delta>0$ represents the noise level. Unless otherwise specified, 10\% noise was added to the synthetic data for the inversion.

Noting that the amplitude of the signal gradually decreases according to the proposed imaging functions\eqref{eq:indicator}, therefore, in practical computations, we only need to utilize measurement data within a finite time. Thus, the imaging functionals \eqref{eq:indicator} in 2D can be truncated by
\begin{equation*}
  \mathcal{I}^{(2)}_{\delta}(z)=\int_{0}^{T} \left| \int_{\Gamma} p^s_{\delta}(x,t+c_0^{-1}|x-z|)\,  \varphi_{\sigma}^{(2)}(x,t,z)  \, \mathrm{d}s(x)\right|^2 \mathrm{d}t, \quad z\in D,
\end{equation*}
where $T>0$ is the terminal time. In what follows, we  use this truncated imaging functional to identify the unknown targets.
Next, we specify discretization details of the above imaging functionals.
Suppose that $N_m$ receivers are equidistantly distributed on the boundary $\Gamma$, i.e., $x_m \in \Gamma$, $m=1,2,..., N_m$. Meanwhile, the observation period is divided into $N_t$ equidistant recording time instants, that is,  $t_n \in(0,\, T]$, $n=1,2,...,N_t$. 
For simplification, the distance between two  adjacent  receivers and the time step are set to be $\Delta_{\Gamma}$ and $ \Delta_T$, respectively. Furthermore, we assume that the sampling domain $D$ is enclosed by the measurement surface $\Gamma$ such that $\Omega \subset D$. In what follows,  we use
a uniformly distributed $N_h \times N_h$ and $N_h\times N_h\times N_h$  sampling mesh $\mathcal{T}_h$ over the sampling domain $D$ for two and three dimensions, respectively. Hence, for each sampling point $z_{\ell}\in \mathcal{T}_h$, $\ell=1,2, ... , N_h^d$,  we compute the discretized imaging functions  with a single source \eqref{eq:indicator} by
\begin{equation*}
  \mathcal{I}_{\delta}^{(2)}(z_{\ell})=\Delta_T \sum_{n=1}^{N_t}\left( \Delta_{\Gamma}\sum_{m=1}^{N_m}  p_{\delta}^s(x_m, t_n+c_0^{-1}|x_m-z_{\ell}|)  \,  \varphi_{\sigma}^{(2)}(x_m,t_n,z_{\ell}) \right)^2,
\end{equation*}
where the test function $\varphi_{\sigma}^{(d)}(x_m,t_n,z_{\ell})$ is the discretized form of \eqref{eq:test}.
For the sake of comparison, the imaging functions are scaled to the range $[0,\, 1]$ as follows
\begin{equation}\label{eq:indicator-normal}
  \widetilde{\mathcal{I}}_{\delta}^{(2)}(z_{\ell})=
  \frac{\mathcal{I}^{(2)}(z_{\ell})}{\max\limits_{z_{\ell} \in \mathcal{T}_h} \mathcal{I}^{(2)}(z_{\ell})}.
\end{equation}
In what follows, time step is chosen as $\Delta_T=0.02$s and the background speed is set to be $c_0=4 $m/s. In addition, a $60 \times 60$ sampling grid is used in the following numerical experiments.

%

\subsection{Reconstruction of point-like scatterers with a single source}

 \begin{figure}
 \begin{minipage}{1\textwidth}
\subfigure[geometry setting]{\includegraphics[width=0.32\textwidth]
                   {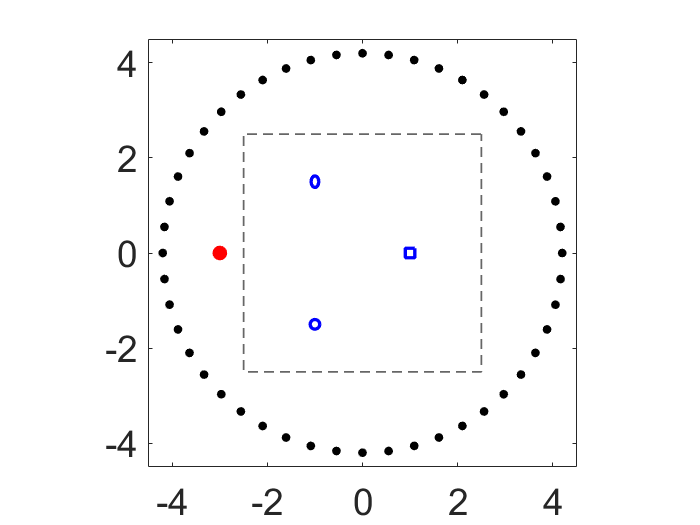}}
\subfigure[Gaussian pulse wave]{\includegraphics[width=0.32\textwidth]
                   {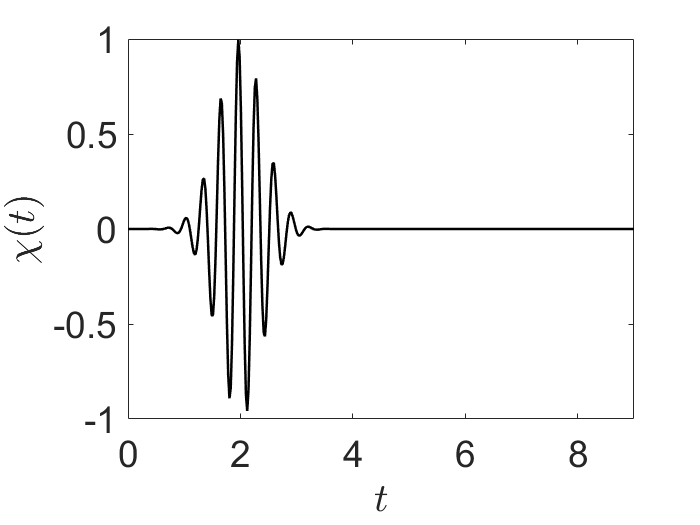}}
\subfigure[Fourier spectrum ]{\includegraphics[width=0.32\textwidth]
                   {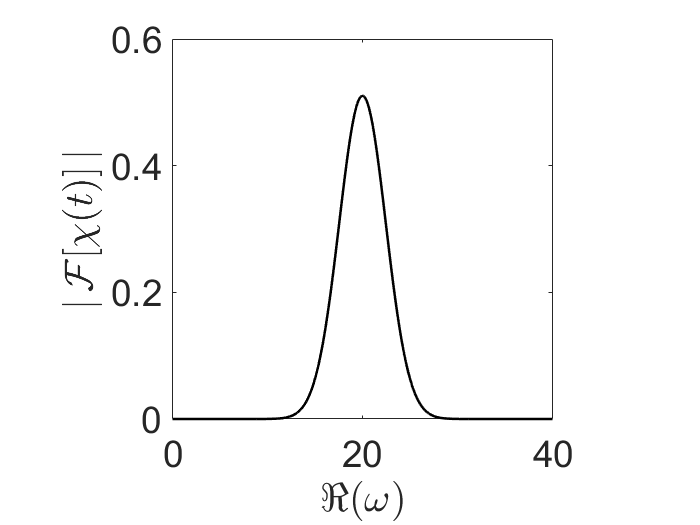}}\\
\caption{\label{fig:point-geometry} Schematic illustration of the time-domain acoustic scattering problem in two dimensions.
(a) Geometry setting of the problem; (b) plotting of Gaussian-modulated sinusoidal pulse wave with $\omega_0=20$, (c) plotting of the spectrum of the corresponding Gaussian-modulated sinusoidal pulse wave via the Fourier transform. \vspace{1cm}
  }
\end{minipage}
\begin{minipage}{1\textwidth}
\subfigure[$\omega_0=5$]{\includegraphics[width=0.32\textwidth]
                   {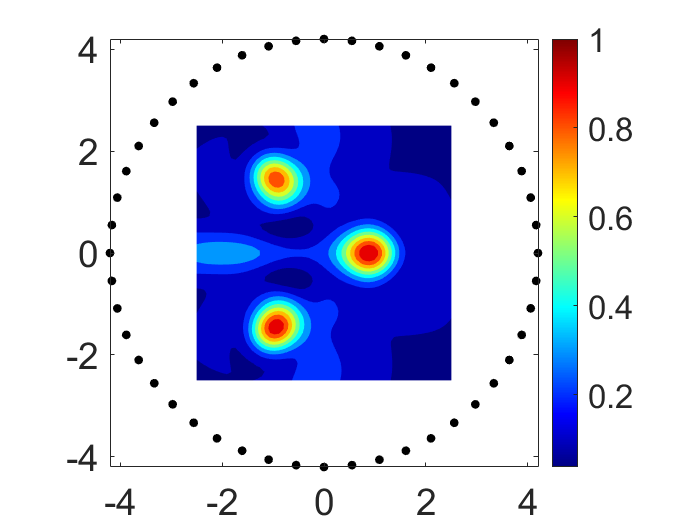}}
\subfigure[$\omega_0=10$]{\includegraphics[width=0.32\textwidth]
                   {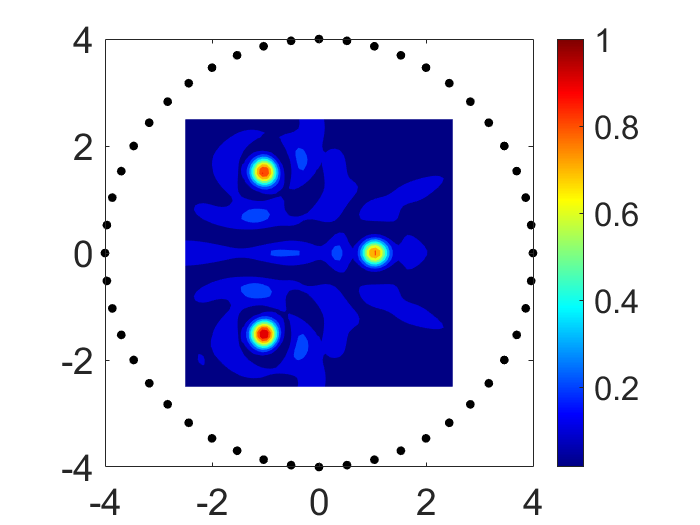}}
\subfigure[$\omega_0=20$]{\includegraphics[width=0.32\textwidth]
                   {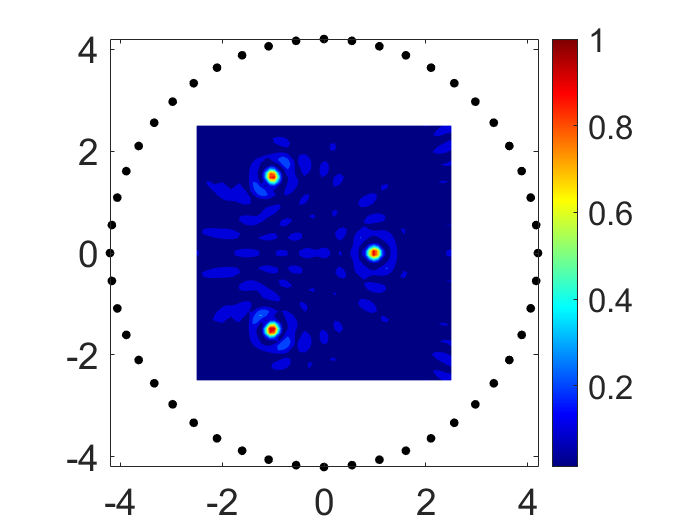}}\\
\subfigure[$\displaystyle \theta\in \left(\frac{\pi}{4}, \frac{7\pi}{4} \right)$]{\includegraphics[width=0.32\textwidth]
                   {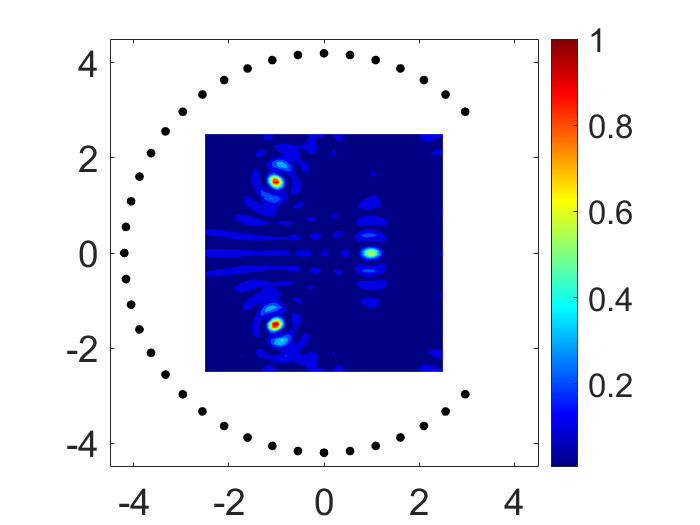}}
\hfill\subfigure[$\displaystyle \theta\in \left(\frac{\pi}{2}, \frac{3\pi}{2} \right)$]{\includegraphics[width=0.32\textwidth]
                   {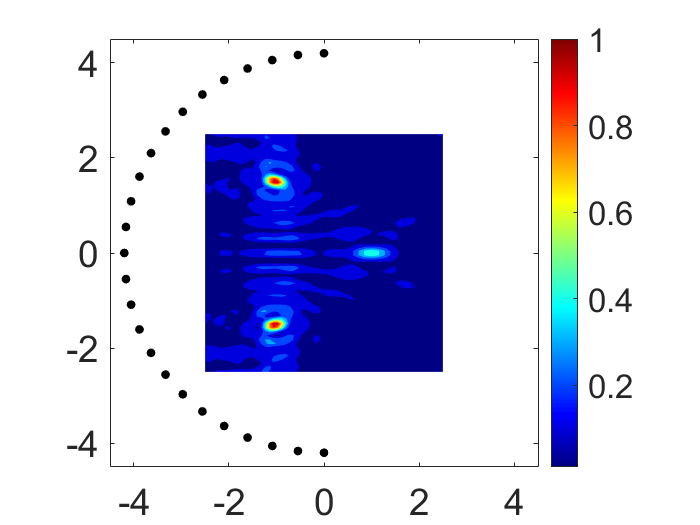}}\hfill
\hfill\subfigure[$\displaystyle \theta\in \left(\frac{3\pi}{4}, \frac{5\pi}{4} \right)$]{\includegraphics[width=0.32\textwidth]
                   {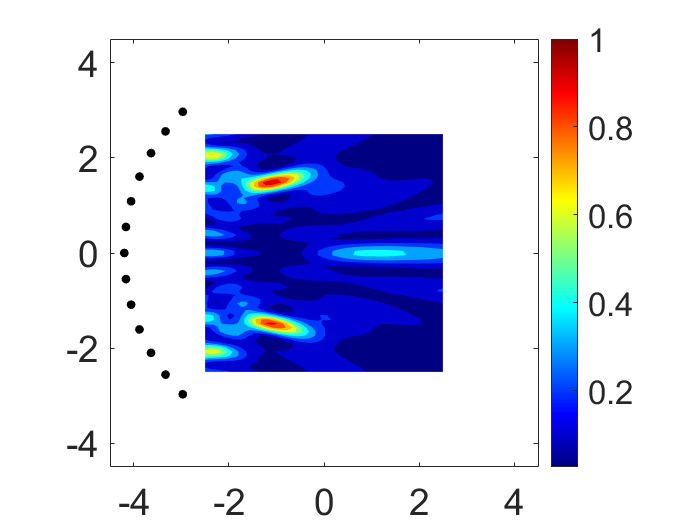}}\hfill\\
\caption{\label{fig:point} Contour plots of the imaging functional $\widetilde{\mathcal{I}}^{(2)}(z_{\ell})$ defined in \eqref{eq:indicator-normal}  by using the Gaussian-modulated sinusoidal pulse wave $\chi_1$ as the incident signal.
Top row: Reconstructions with different center frequencies $\omega_0$;
bottom row: reconstructions with different observation aperture $\theta$. }
\end{minipage}
\end{figure}

In this part, we aim to reconstruct multiple point-like scatterers by a single monopole source with different incident signals.
In what follows, we consider three kinds of  causal temporal signals as the incident source, that is, Gaussian-modulated sinusoidal pulse wave, smooth sawtooth wave and tempered sinusoidal wave.

We assume that the single source is located at $(-3,0)$ and $48$ receivers are distributed on the observation circle with radius $r=4.2$. Moreover, the sampling domain is chosen as $D=[-2.5,2.5]\times[-2.5,2.5]$.  Figure \ref{fig:point-geometry}(a) presents the two-dimensional geometrical setting of  the time-domain acoustic scattering problem,
where the location of single incident source is marked by the red point,
 the locations of receivers are marked by the black points,
 the  point-like scatterers are plotted as the blue curves and the sampling domain is marked by the dotted square.

Firstly, we consider the Gaussian-modulated sinusoidal pulse wave as the incident signal, which is given by
 \begin{equation*}
 \chi_1(t):=
   \begin{cases}
   \displaystyle \sin(\omega_0 t) \, \exp\left(-3(t-2)^2 \right) , & t\geq 0, \medskip \\
   0,  & t<0,
   \end{cases}
 \end{equation*}
 where $\omega_0\in \mathbb{R}_+$ denotes the center frequency. Figure \ref{fig:point-geometry} (b) and (c) respectively show the waveform and wavenumeber spectrum of the above Gaussian-modulated sinusoidal pulse wave with $\omega_0=20$. Here the  terminal time is given by $T=6$s and $450$ recording time steps are utilized. Assume that the true scatterers consists of a disk with radius $0.1$ located at $(-1,-1.5)$, a square with sidelength $0.1$ located at $(1, 0)$, and an elliptic with center $(-1,1.5)$, semi-major axis $0.12$ and semi-minor axis $0.08$. The sound speed $c(x)$ inside the disk, square and elliptic are $15$, $30$ and $10$, respectively.  Noting that the incident wave is  a  Gaussian-modulated sinusoidal pulse wave,   as discussed in remark \ref{rem:sigma}, we can  set $\sigma=0$ in this example.

\begin{figure}
\hfill\subfigure[geometry setting]{\includegraphics[width=0.33\textwidth]
                   {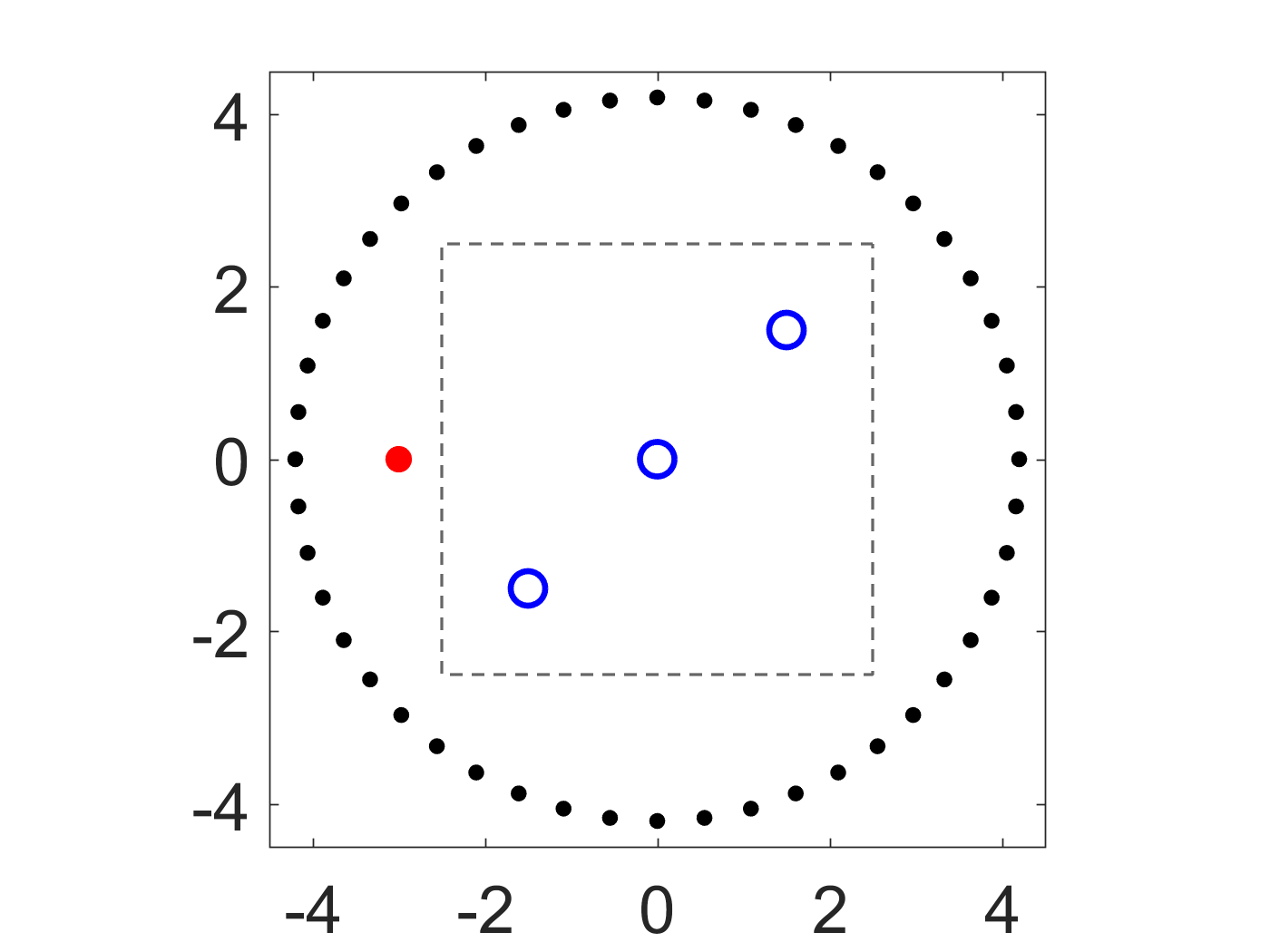}}\hfill
\hfill\subfigure[smooth sawtooth wave]{\includegraphics[width=0.33\textwidth]
                   {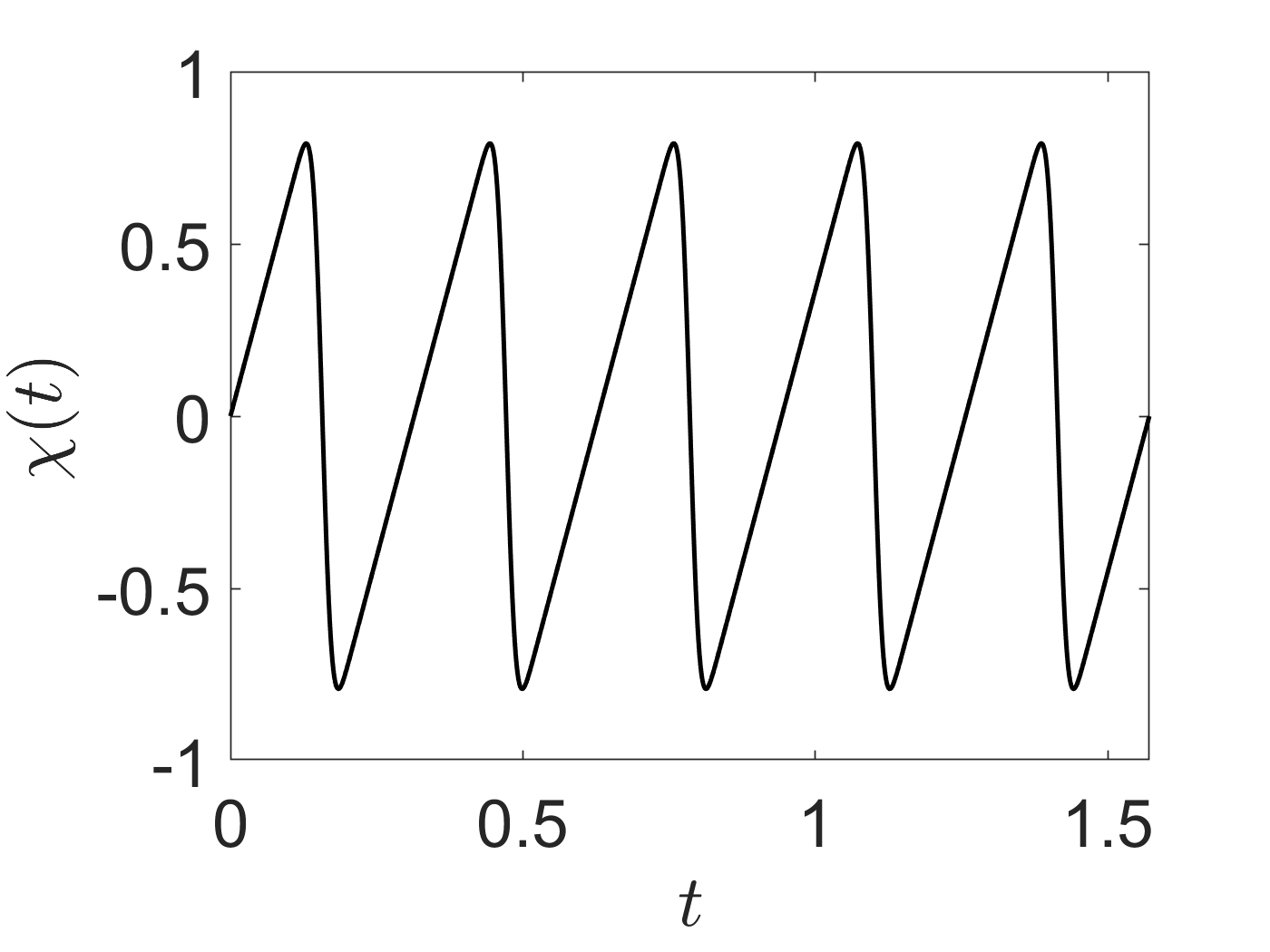}}\hfill
\hfill\subfigure[$T=1$]{\includegraphics[width=0.33\textwidth]
                   {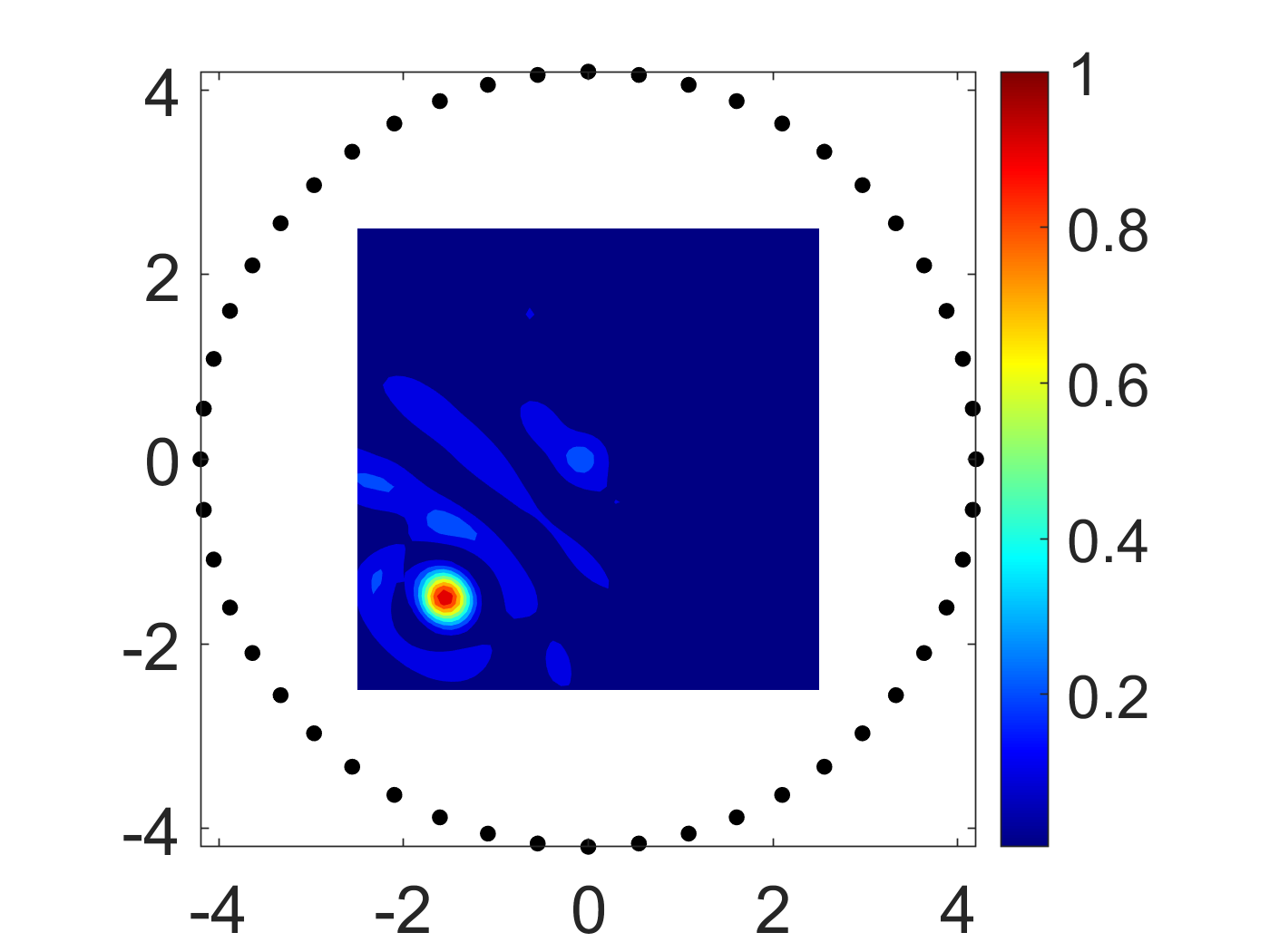}}\hfill\\
\hfill\subfigure[$T=2$]{\includegraphics[width=0.33\textwidth]
                   {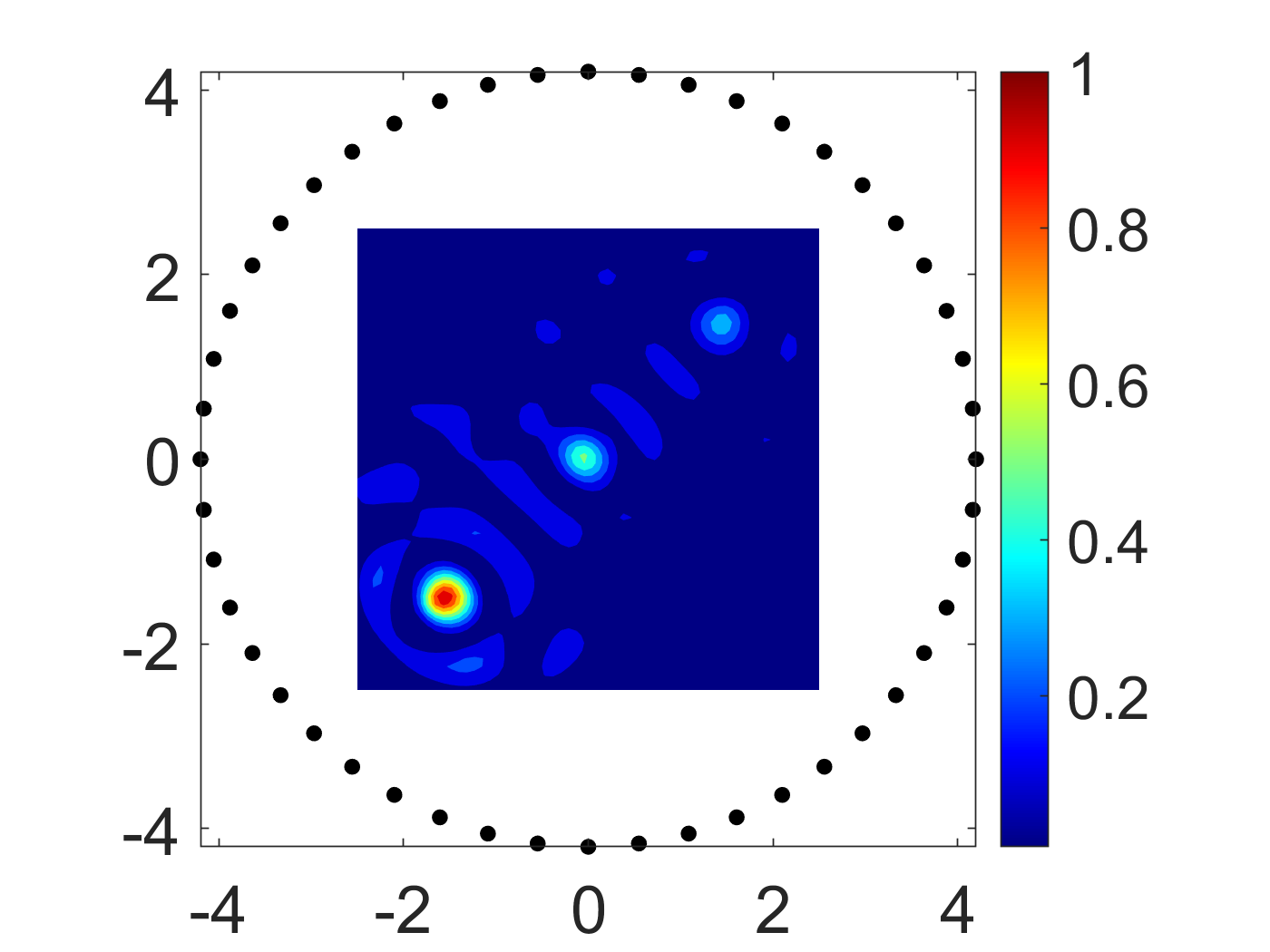}}\hfill
\hfill\subfigure[$T=4$]{\includegraphics[width=0.33\textwidth]
                   {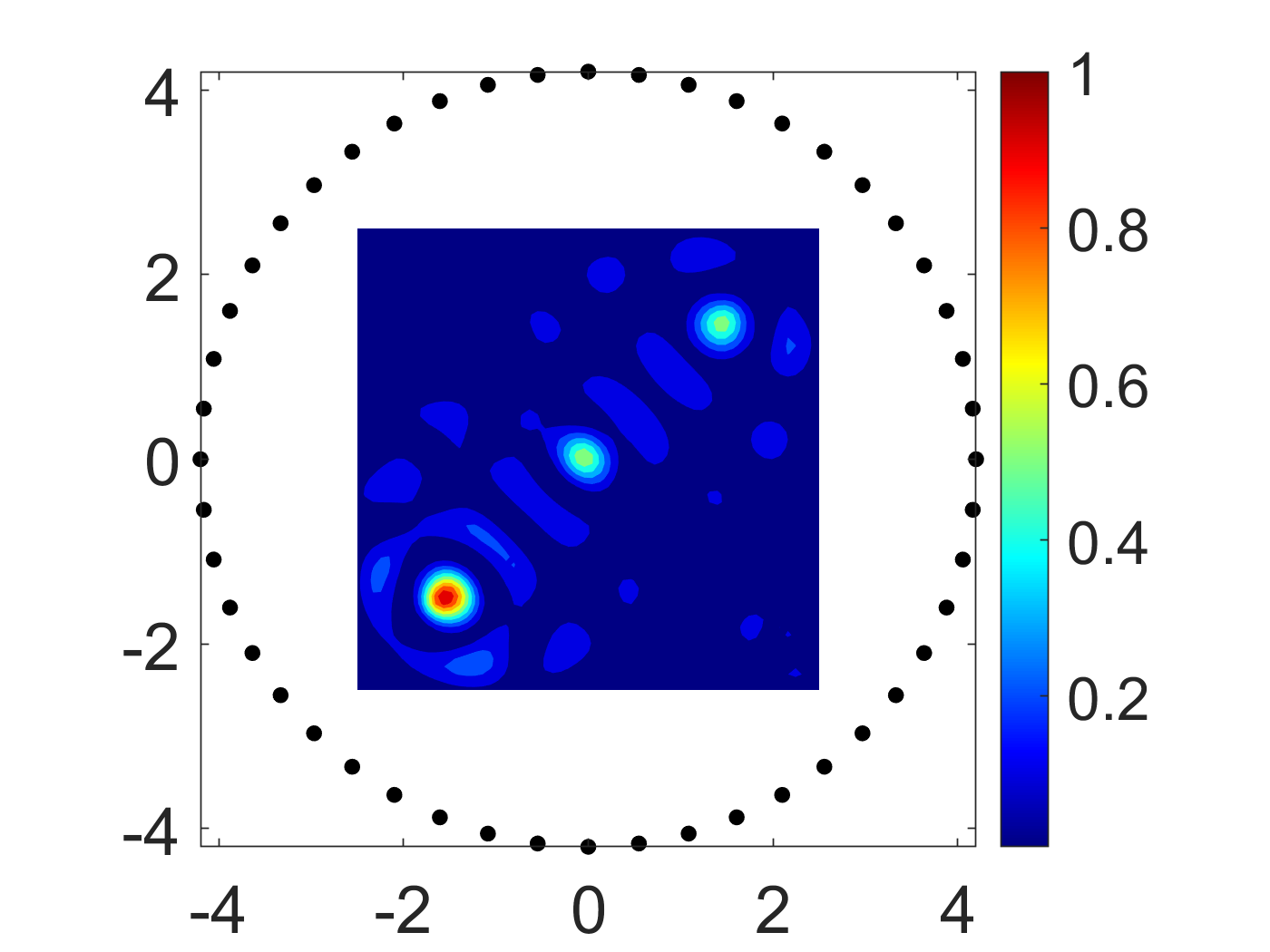}}\hfill
\hfill\subfigure[$T=6$]{\includegraphics[width=0.33\textwidth]
                   {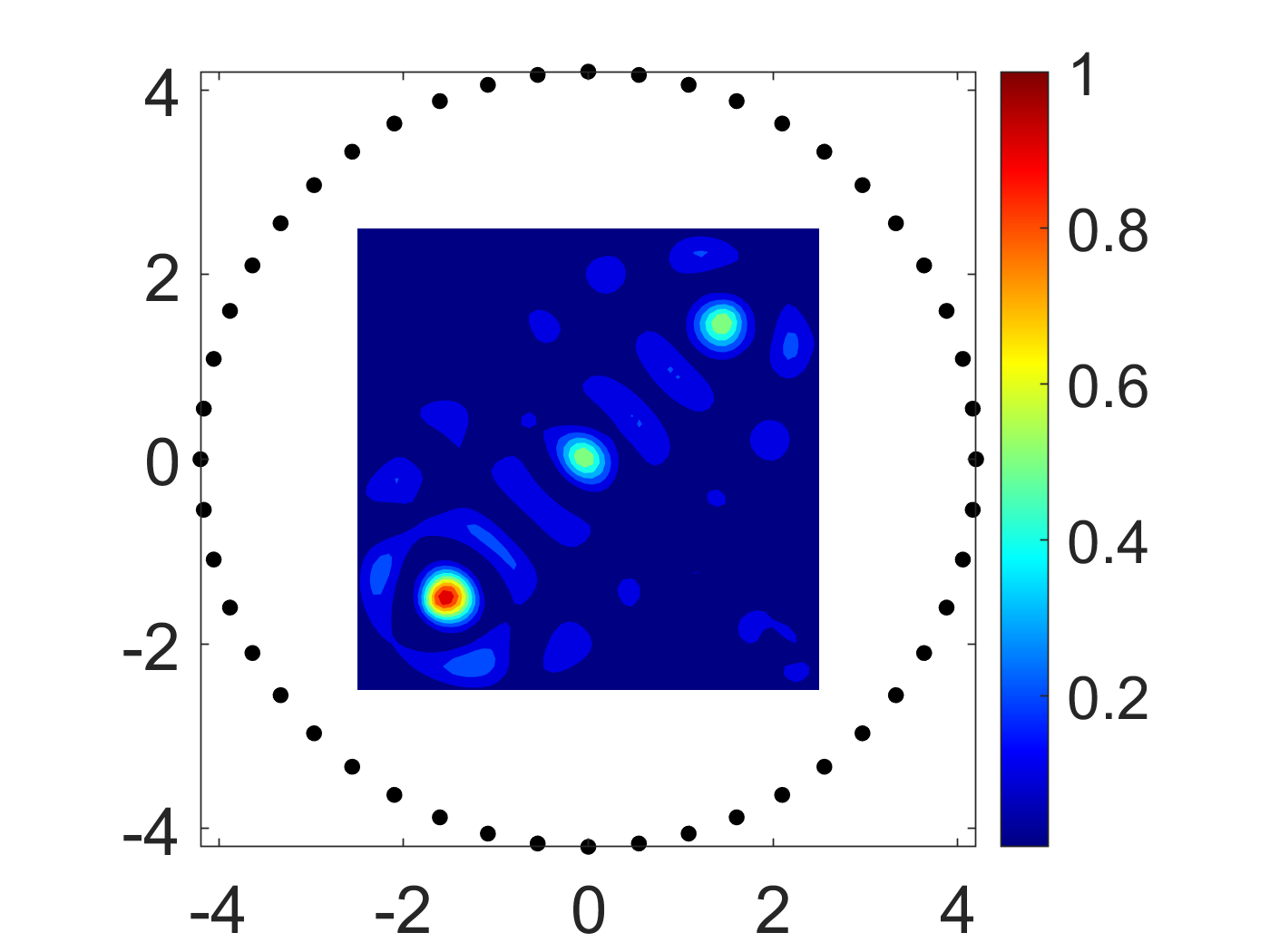}}\hfill
\caption{\label{fig:sawtooth} Contour plots of the imaging functional $\widetilde{\mathcal{I}}^{(2)}(z_{\ell})$ defined in \eqref{eq:indicator-normal}  by using the smooth sawtooth wave $\chi_2$ as the incident signal. (a): Geometry setting of the problem; (b) plotting of the smooth sawtooth wave; (c)-(f): reconstructions  with different
terminal times $T$.
 }
\end{figure}

Figure \ref{fig:point} presents the contour plots of the imaging functional $\widetilde{\mathcal{I}}_{\delta}^{(2)}(z)$ by using the Gaussian pulse as the incident wave with different center frequencies.
It can be seen that the the imaging functional attains a significant local maximum near the exact locations of the point-like scatterers. In comparison of the imaging results among figure \ref{fig:point}(a)-(c), one can observe that the resolution of the reconstruction is significant improved as $\omega_0$ increases. Furthermore, from figure \ref{fig:point}(c)-(f), we can find that the point-like scatterers is less precisely captured as the aperture decreases.

Secondly, we consider a smooth sawtooth wave as the incident wave, which is parameterized as
 \begin{equation*}
   \chi_2(t):=
   \begin{cases}
   \displaystyle \int_{-\infty}^{+\infty}\!\!
      \sqrt{\frac{3000}{\pi}}\! \left( \frac{20 \tau+\pi}{2\pi}- \! \left\lfloor\frac{20\tau+\pi}{2 \pi}\right \rfloor \!-\frac{1}{2} \right)\! \exp\! \left(-3000(t-\tau)^2\right) \, \mathrm{d}\tau, & t\geq 0, \medskip \\
   0,  & t<0,
   \end{cases}
 \end{equation*}
where $\lfloor X \rfloor$ denotes the greatest integer less than or equal to $X$.  Here the  terminal time is given by $T=3$s and $500$ recording time steps are received.  Suppose that the unknown scatterer consists of three same disk with radius $0.2$ and they are located at $(-1.5,-1.5)$, $(0,0)$ and $(1.5,1.5)$. The sound speed $c(x)$ inside the disks are $10$, $20$ and $40$, respectively.

Figure \ref{fig:sawtooth}(b) plots the truncated waveform of the smooth sawtooth  wave. In this numerical example, the parameter $\sigma$ is chosen as $0.2$. As shown in figure \ref{fig:sawtooth}(c)-(f), it is clear to see that the small scatterers are determined one by one as the terminal time increases. The main reason is that we first receive  the scattered field generated by the point-like scatterer that is closest to the location of the  incident source.

\begin{figure}
\hfill\subfigure[geometry setting]{\includegraphics[width=0.33\textwidth]
                   {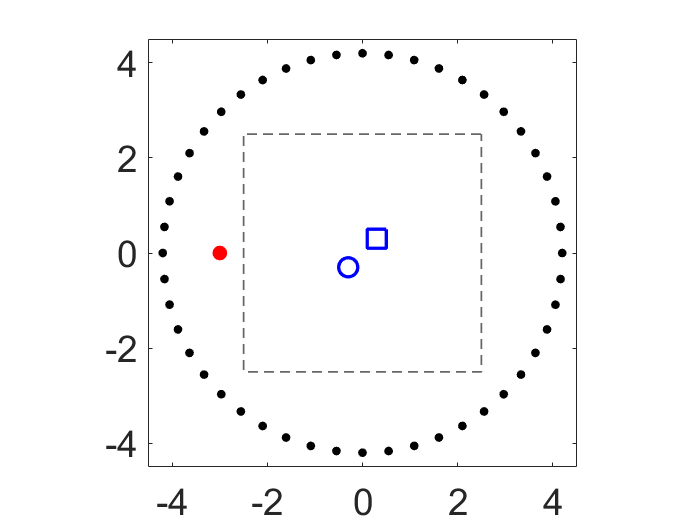}}\hfill
\hfill\subfigure[]{\includegraphics[width=0.33\textwidth]
                   {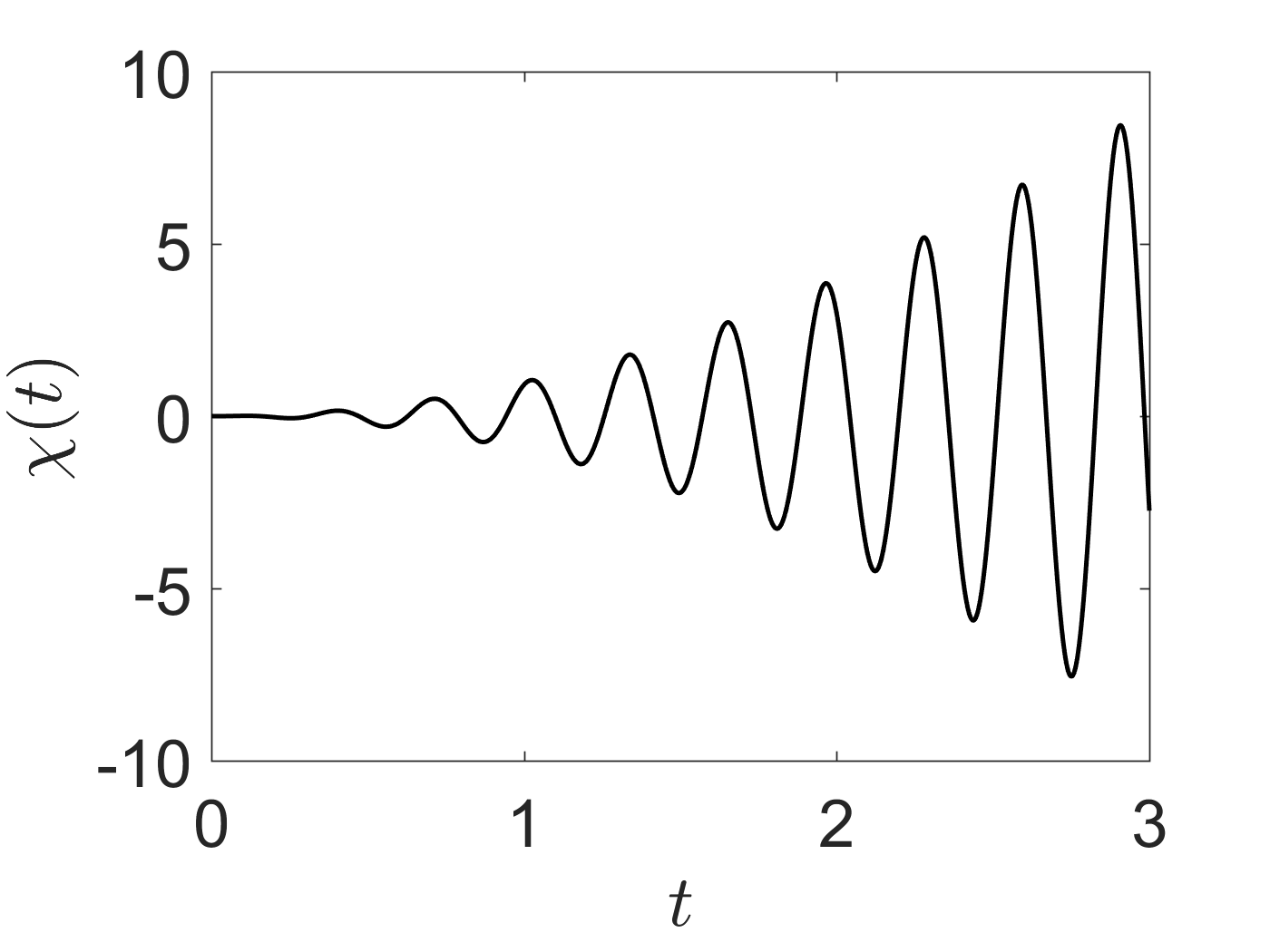}}\hfill
\hfill\subfigure[$\sigma=0$]{\includegraphics[width=0.33\textwidth]
                   {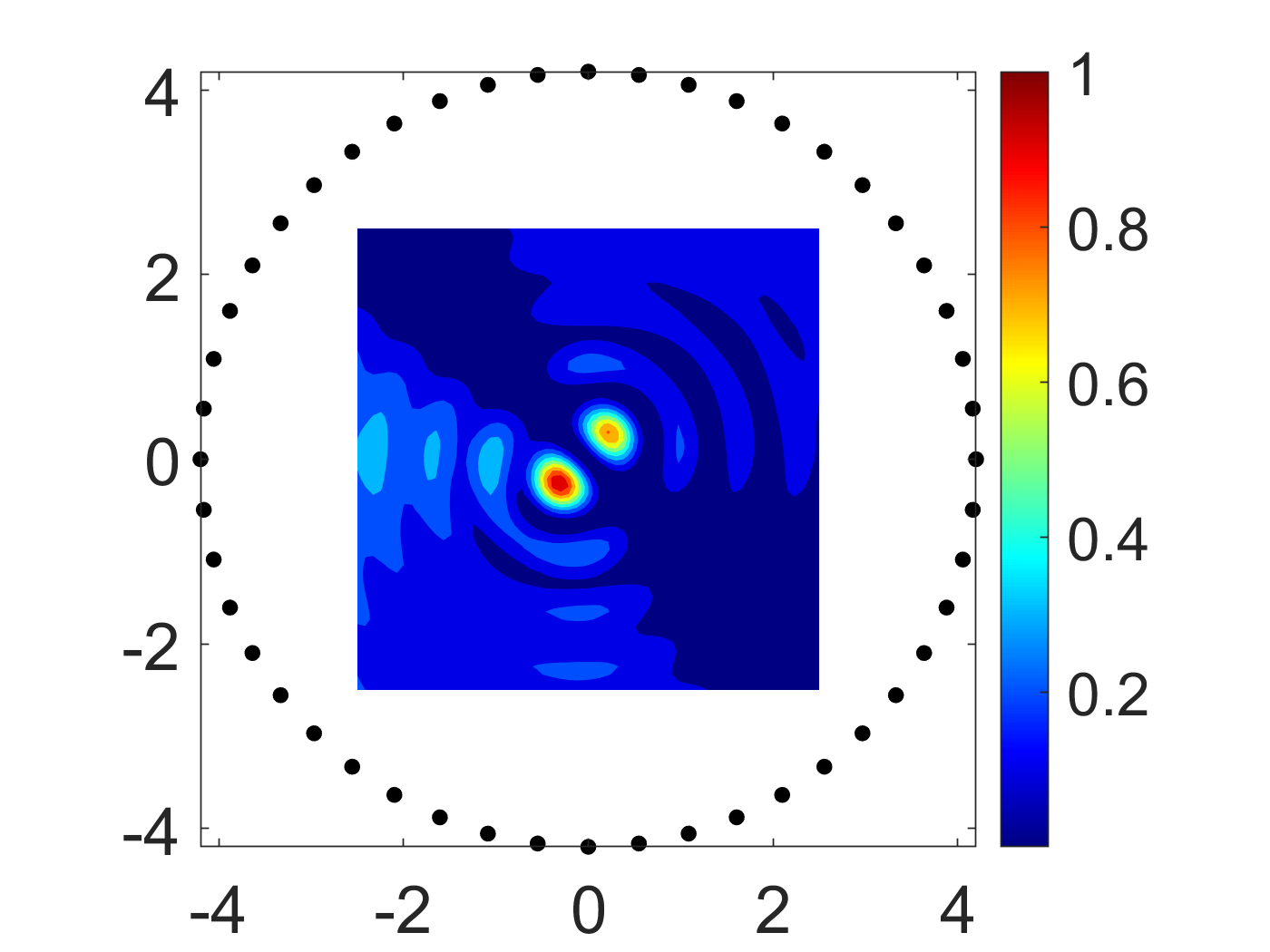}}\hfill\\
\hfill\subfigure[$\sigma=3$]{\includegraphics[width=0.33\textwidth]
                   {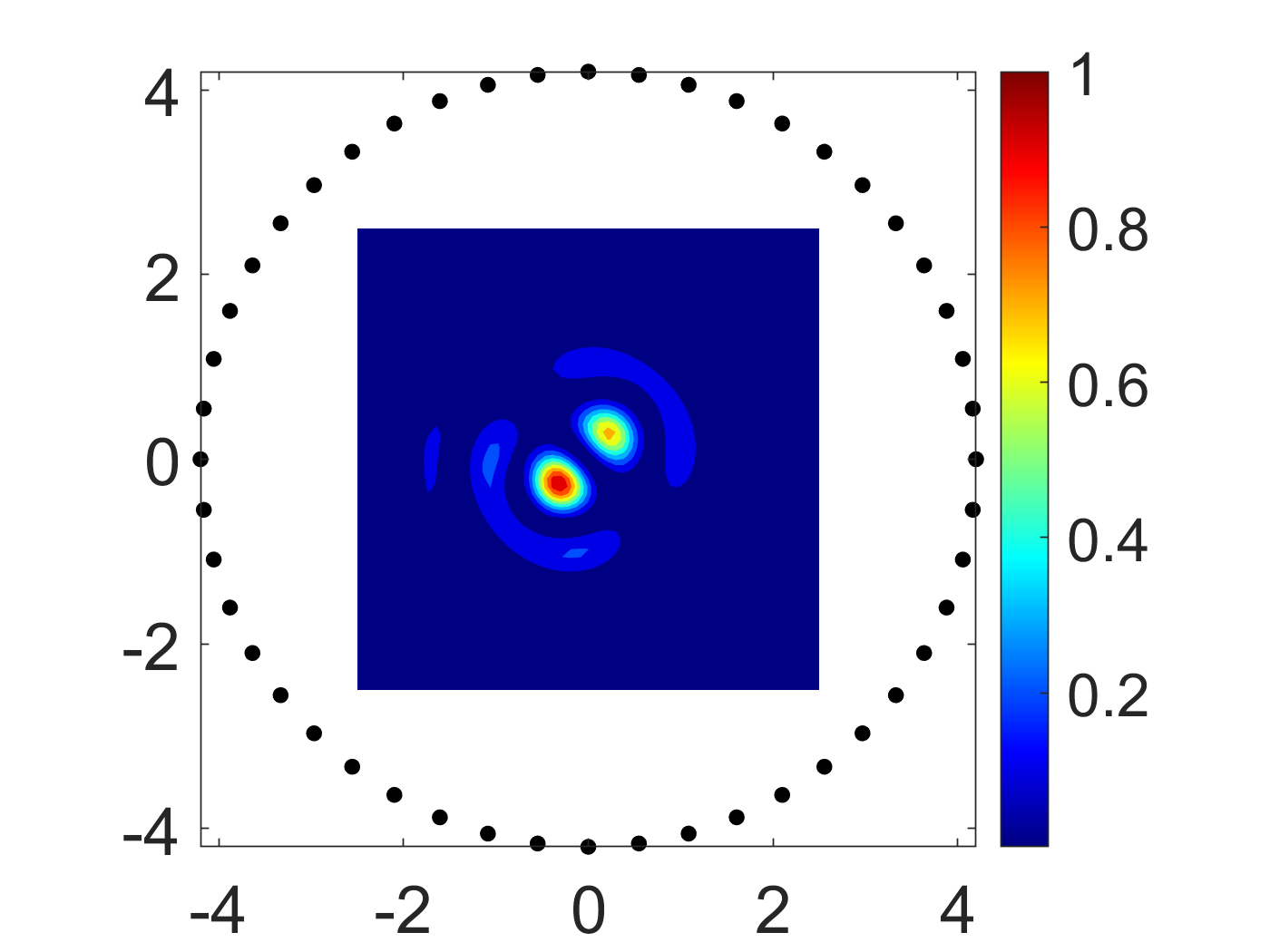}}\hfill
\hfill\subfigure[$\sigma=6$]{\includegraphics[width=0.33\textwidth]
                   {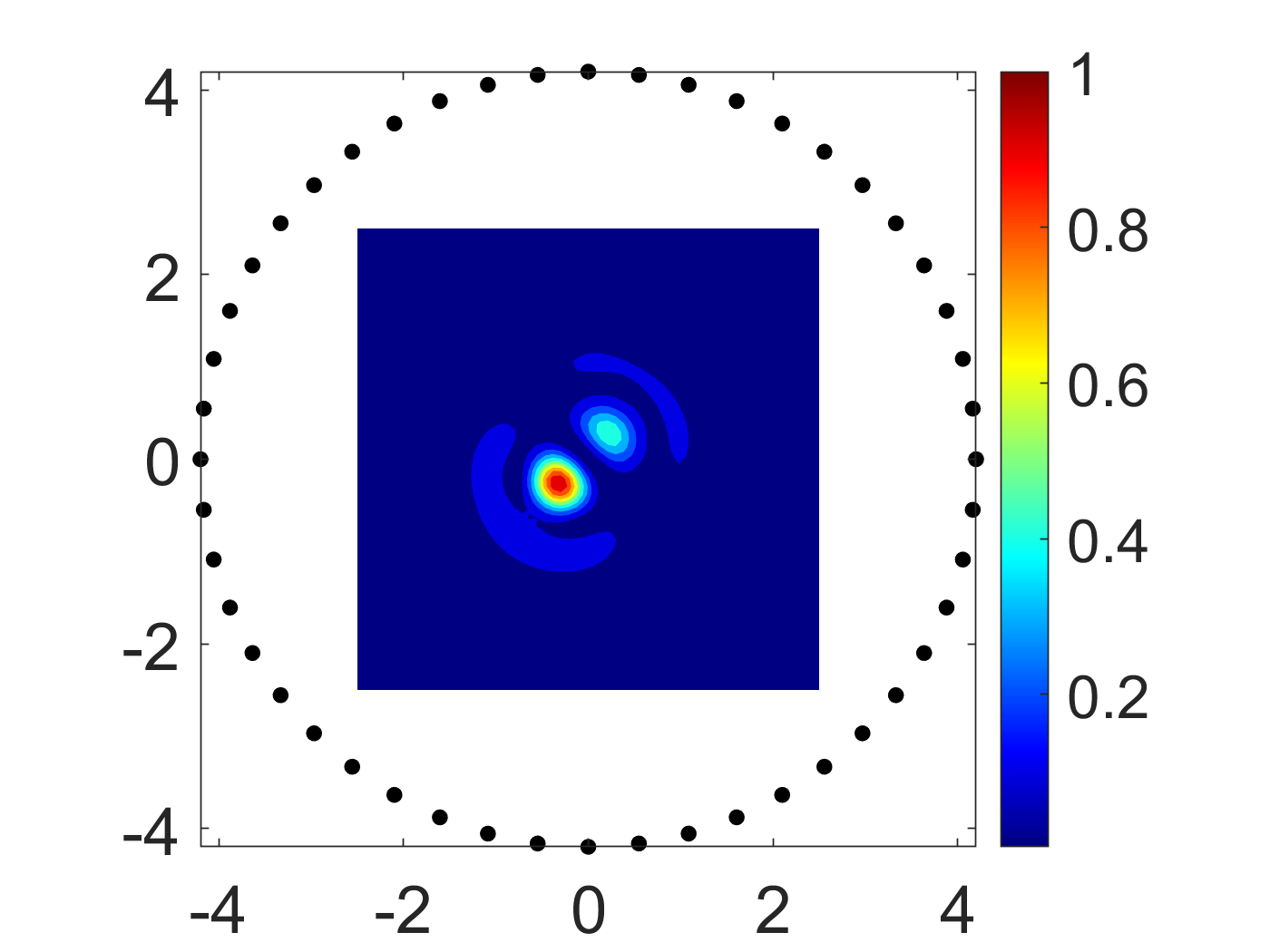}}\hfill
\hfill\subfigure[$\sigma=9$]{\includegraphics[width=0.33\textwidth]
                   {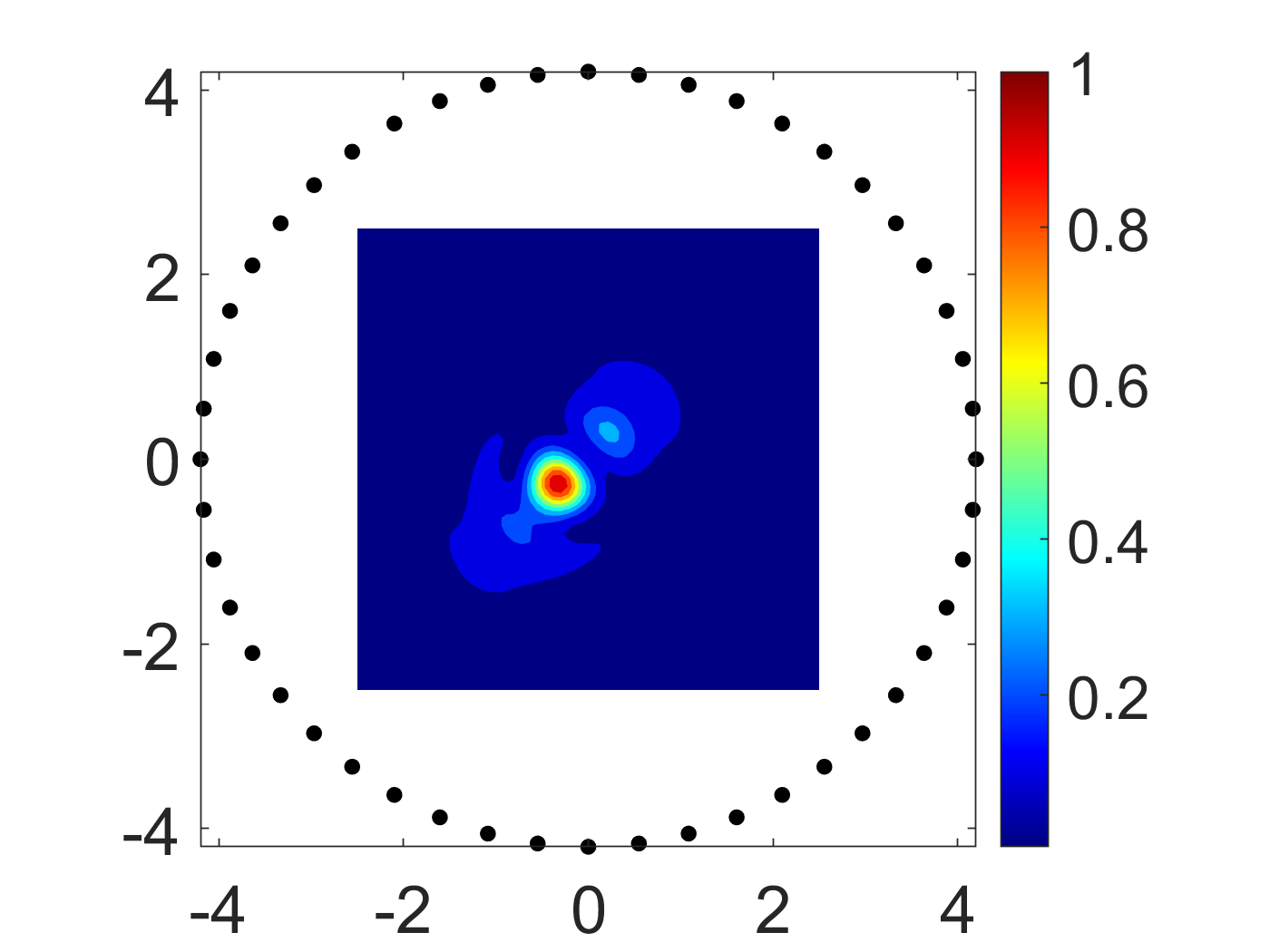}}\hfill\\
\caption{\label{fig:tempered} Contour plots of the imaging functional $\widetilde{\mathcal{I}}^{(2)}(z_{\ell})$ defined in \eqref{eq:indicator-normal}  by using the tempered sinusoidal wave $\chi_3$ as the incident signal. (a): Geometry setting of the problem; (b) plotting of the tempered sinusoidal wave; (c)-(f): reconstructions with  different
parameters $\sigma$.}
\end{figure}

Finally, we consider a more challenging case. The incident wave is represented by a tempered sinusoidal wave function, that is,
 \begin{equation*}
 \chi_3(t):=
   \begin{cases}
   \displaystyle t^2 \sin(20\,t),  & t\geq 0, \medskip \\
   0,  & t<0.
   \end{cases}
 \end{equation*}
Artificial scattered fields were collected with $T=10$s and $500$ recording time steps. The true scatterers consists of a disk with radius $0.2$ located at $(-0.3,-0.3)$ and a square with length $0.4$ located at $(0.3, 0.3)$.
The sound speed $c(x)$ inside these two domains is given by $20$ and $25$, respectively. In Figure \ref{fig:tempered}, we compare the reconstructed  results using different parameters $\sigma$. It illustrates that the resolution is dependent on the parameter $\sigma$, which is in agreement with the theoretical analysis.

\subsection{Reconstruction of an extended scatterer with multiple sources}

In this part, we aim to reconstruct an extended scatterer by the proposed time-domain direct sampling method. The multi-source imaging functionals are given by
\begin{equation*}\label{eq:indicator-multi}
  \mathcal{I}^{(2)}(z)=\int_{0}^{T} \int_{\Gamma_s} \left| \int_{\Gamma} p^s(x,x_s, t+c_0^{-1}|x-z|)\,  \varphi_{\sigma}^{(2)}(x, t,z)  \, \mathrm{d}s(x)\right|^2 \mathrm{d}s(x_s) \mathrm{d}t, \quad z\in D,
\end{equation*}
where $\Gamma_s$ denotes a measurement surface such that the sources are away from the target objects and  $\varphi_{\sigma}^{(d)}$ is  given by
\begin{equation*}
 \varphi_{\sigma}^{(d)}(x,t,z)=
  \begin{cases}
  \displaystyle \frac{\mathrm{e}^{-\sigma (t+ c_0^{-1}|x-z|)}}{\sqrt{8\pi c_0^{-1}|x-z|}}, & d=2, \bigskip\\
 \displaystyle \frac{\mathrm{e}^{-\sigma(t+ c_0^{-1}|x-z|)}}{4\pi|x-z|}, & d=3.
  \end{cases}
\end{equation*}
In the following examples,
we consider a kite-shaped scatterer
(see figure \ref{fig:kite} (a)), which is parameterized as
\begin{equation*}
  x(\theta)=( \cos\theta+0.65\cos 2\theta-0.65, \ 1.5 \sin\theta), \quad \theta\in[0,\, 2\pi],
\end{equation*}
and a pear-shaped scatterer (see figure \ref{fig:pear} (a)), which is parameterized as
\begin{equation*}
  x(\theta)=(1.6 + 0.24 \cos 3\theta)( \cos\theta,\ \sin\theta), \quad \theta\in[0,\, 2\pi],
\end{equation*}
 We assume that there are $24$ receivers placed on the boundary of a square with sidelength $10$. The sampling domain is chosen as $D=[-4,4]\times[-4,4]$. Figure \ref{fig:kite}(a) presents the geometrical setting of  the problem. Here the locations of the incident source are marked by a red circle, the locations of receivers are marked by  black points and the boundary of the scatterer is plotted by the blue curve.

 To begin with, we consider the reconstruction of a kite-shaped scatterer by using the smooth sawtooth wave, which is generated by single incident wave source.
 Here the parameter $\sigma$ is set to be $0.1$, the terminal time is $T=8$s and $400$ recording time is obtained.
 Unfortunately, it is hardly to determine the shape of the extended target, see figure \ref{fig:kite}(b).
 To overcome this difficulty, we employ multiple incident sources to capture more geometrical details of the target.
 In figure \ref{fig:kite-multi-sources}, we show the reconstructed results with different number of incident sources.
 As expected, reconstructions become more accurate when more transmitter locations are used.

Finally, we consider the reconstruction of a pear-shaped scatterer by using the tempered sinusoidal wave. The terminal time is $T=6$s and $300$ recording time is obtained. From figure \ref{fig:pear}, we can find that our method performs well for the extended scatterer when a suitable parameter $\sigma$ is selected. The optimal choice of $\sigma$ is still unclear and this will be further investigated in our future work.

\begin{figure}
\centering
\subfigure[]{\includegraphics[width=0.45\textwidth]
                   {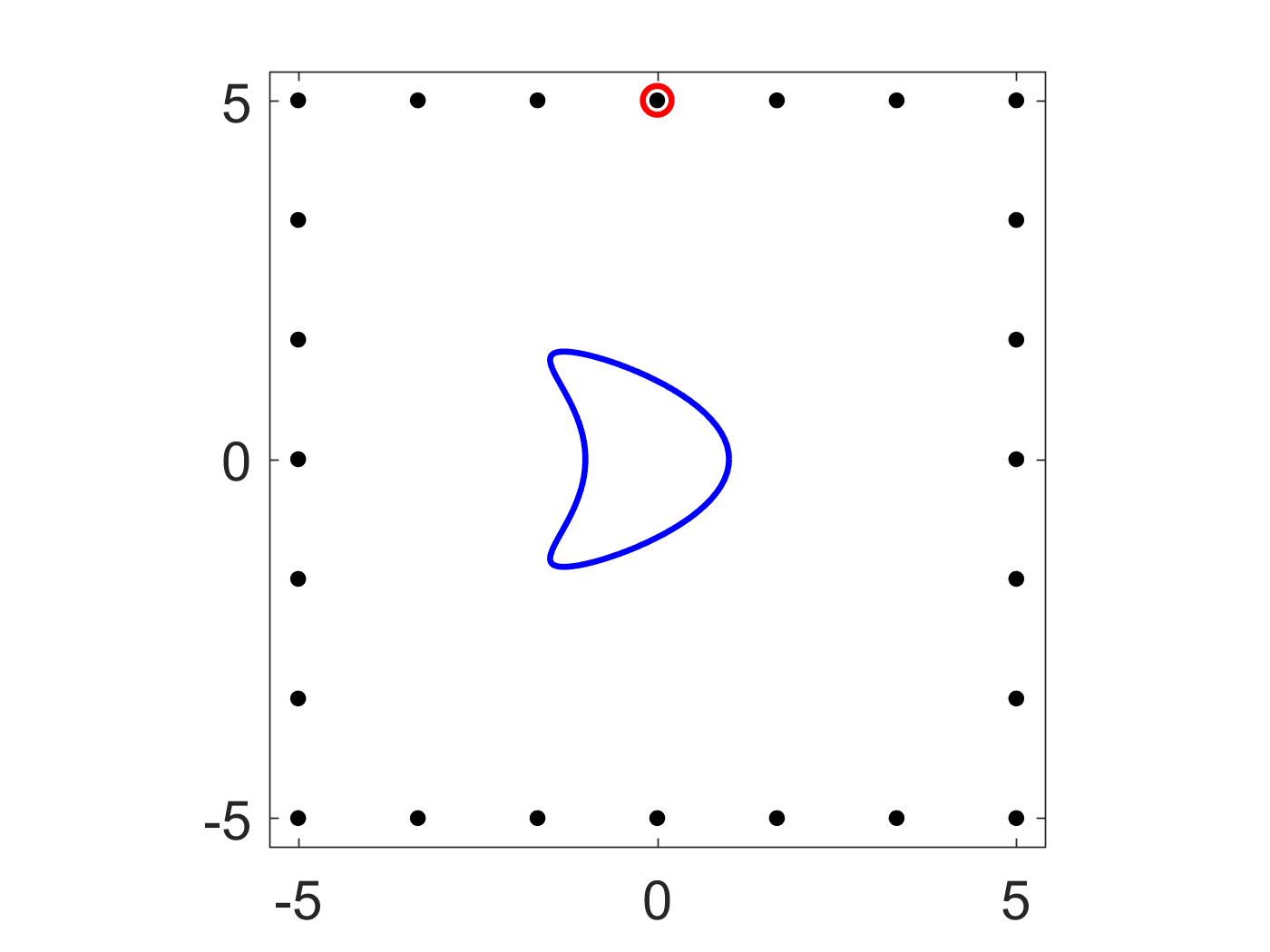}}
\subfigure[]{\includegraphics[width=0.45\textwidth]
                   {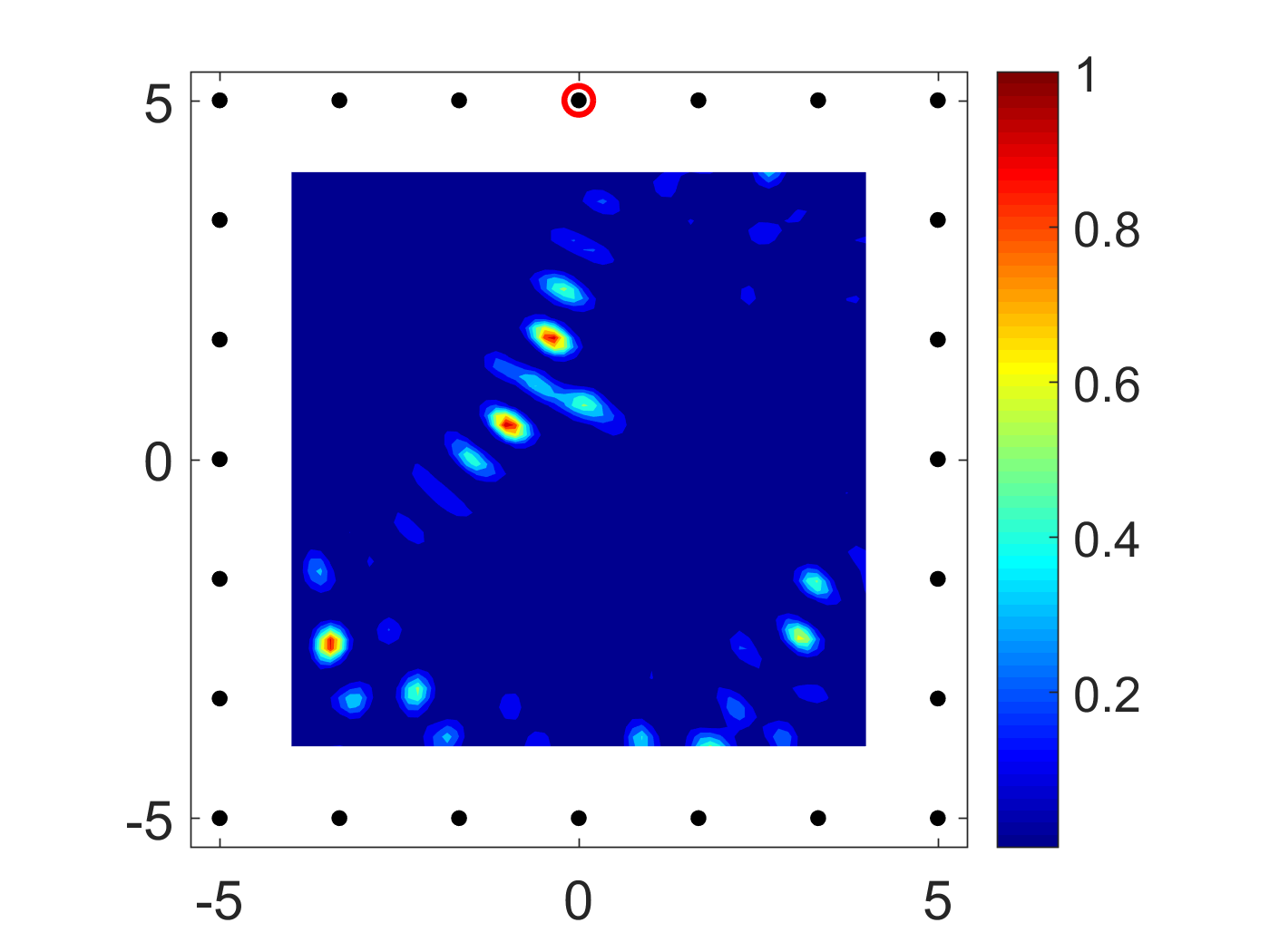}}
\caption{\label{fig:kite} (a): geometry setting;
 (b): reconstruction of the kite-shaped scatterer with one incident source by using the smooth sawtooth wave $\chi_2$. }
\end{figure}

\begin{figure}
\centering
\subfigure[]{\includegraphics[width=0.45\textwidth]
                   {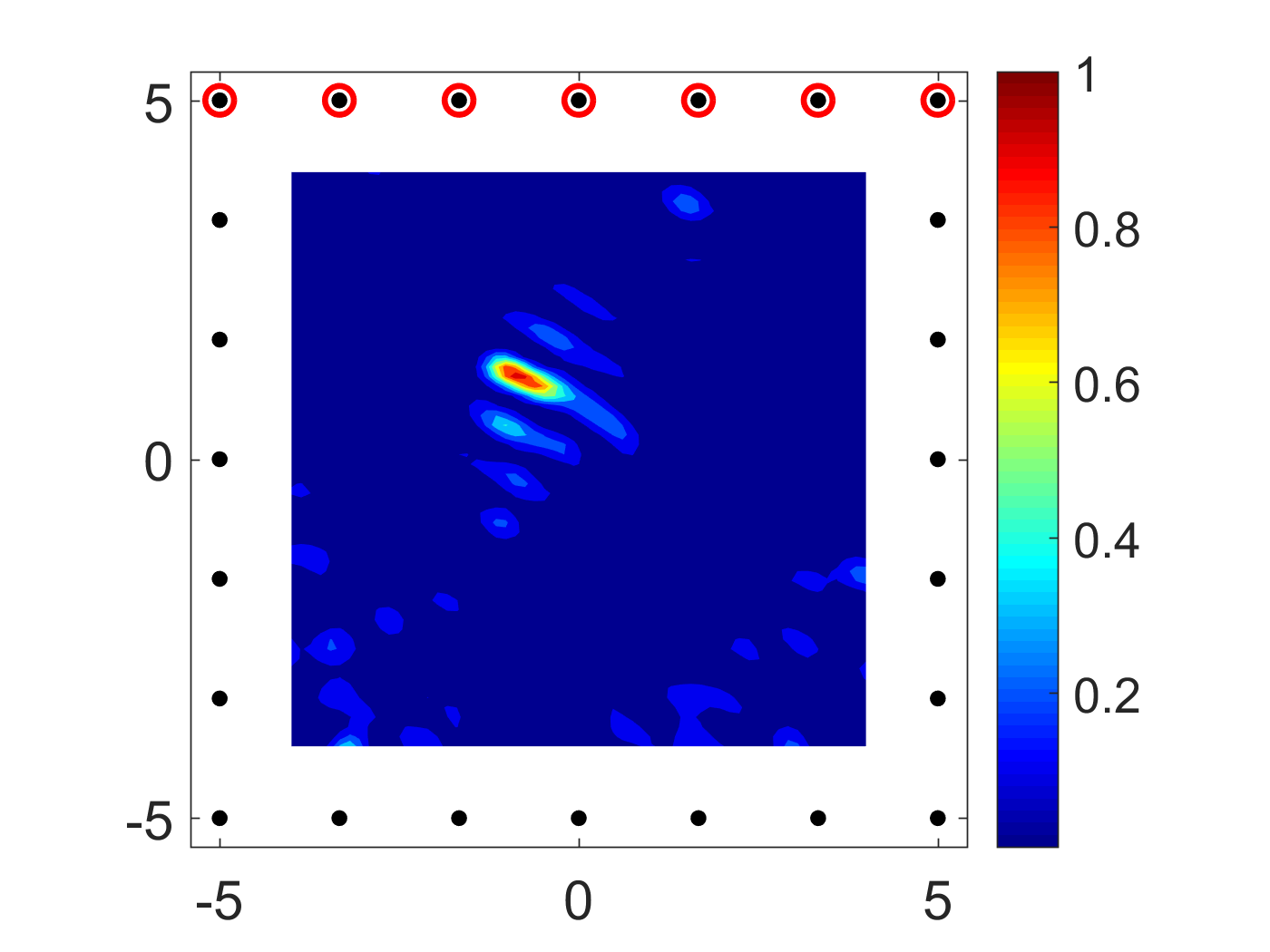}}
\subfigure[]{\includegraphics[width=0.45\textwidth]
                   {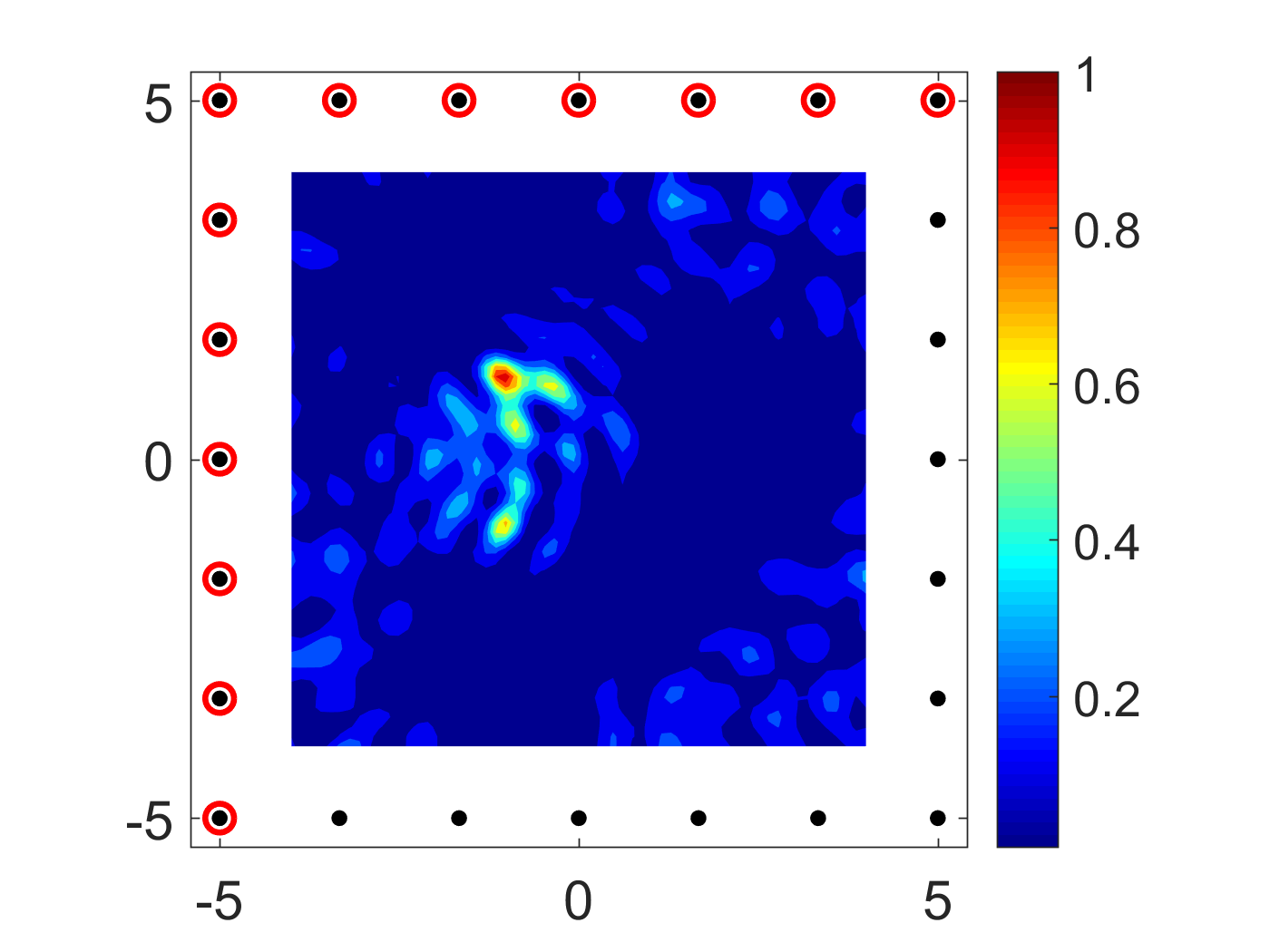}}\\
\subfigure[]{\includegraphics[width=0.45\textwidth]
                   {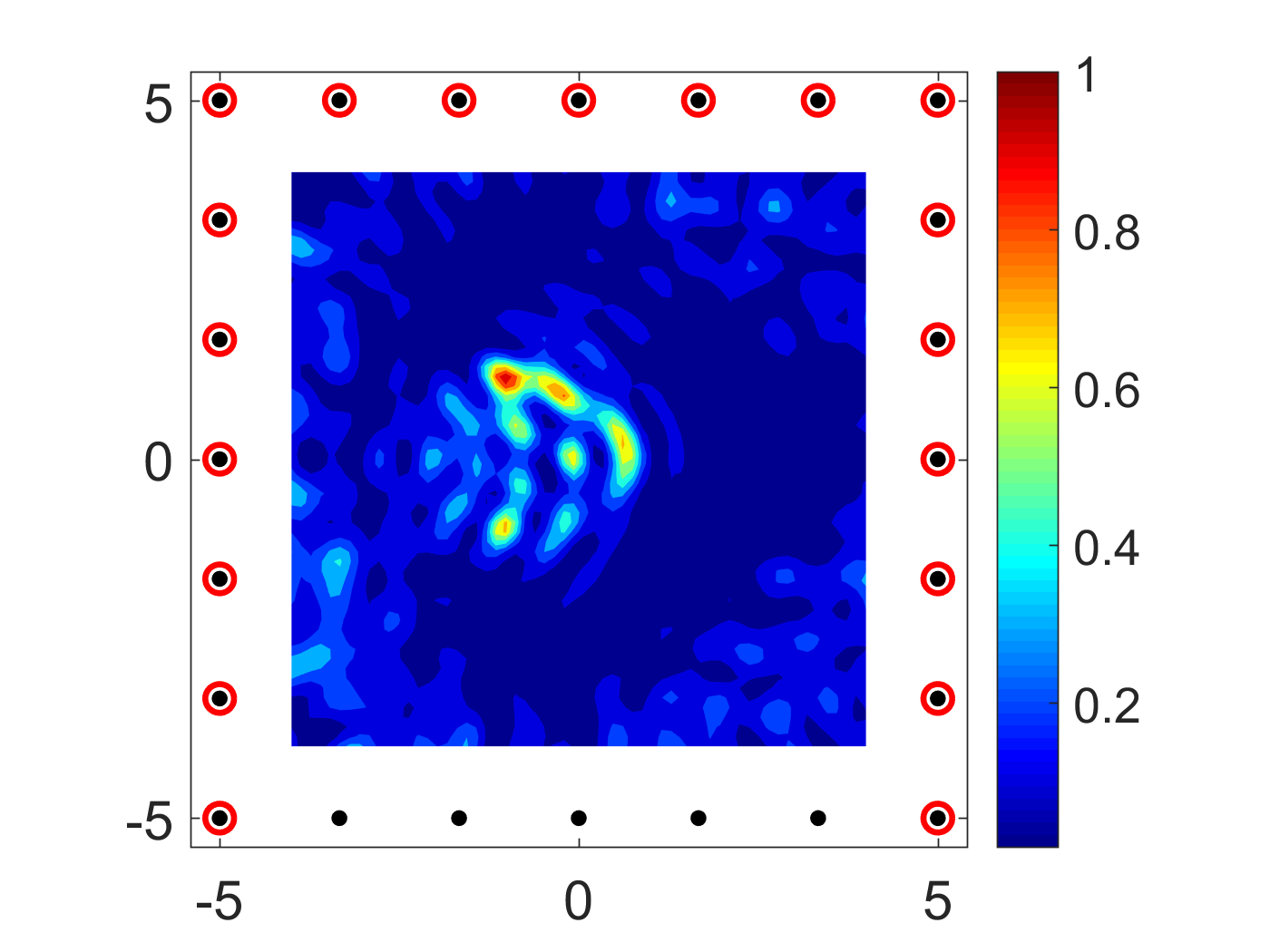}}
\subfigure[]{\includegraphics[width=0.45\textwidth]
                   {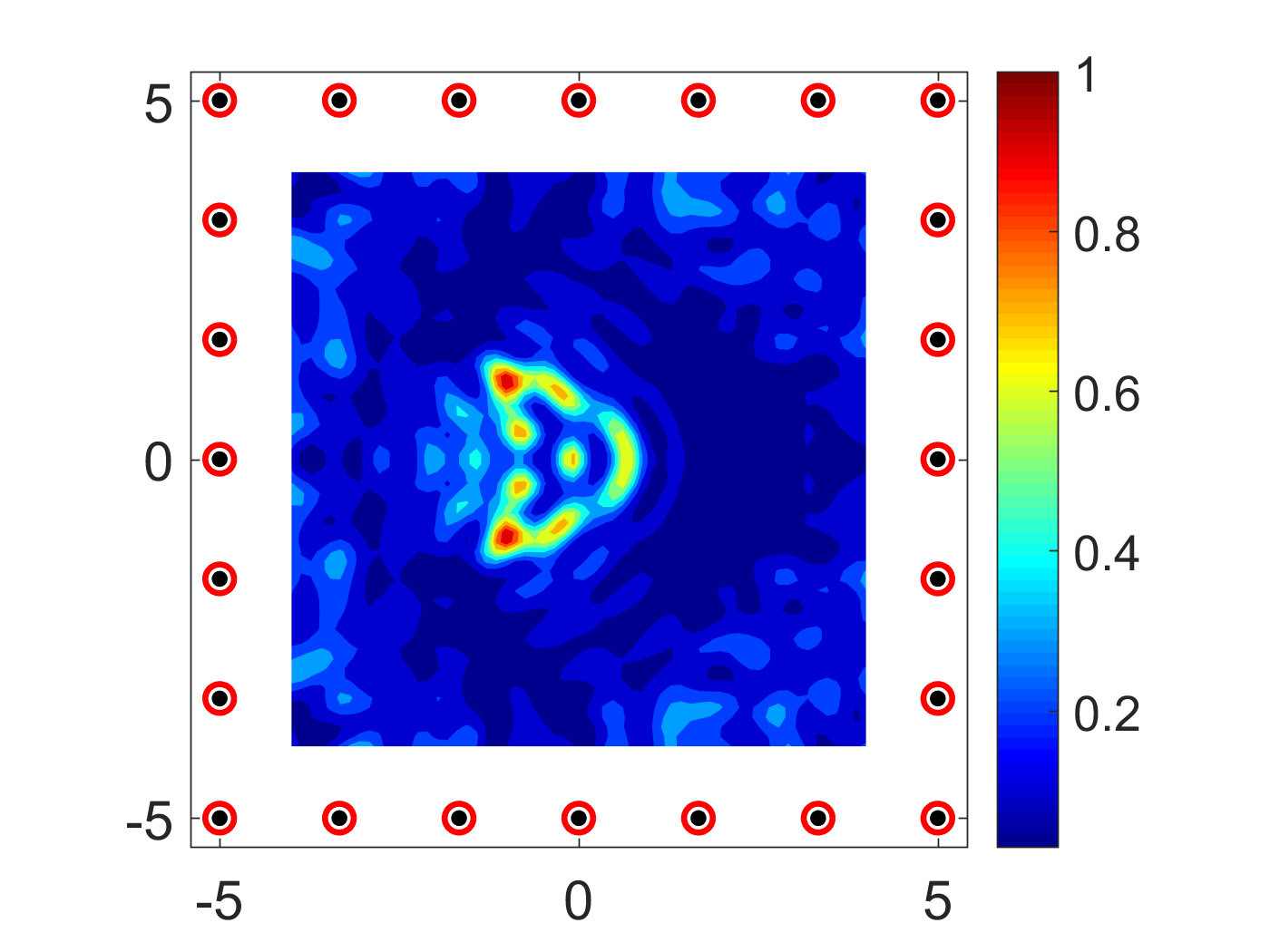}}\\
\caption{\label{fig:kite-multi-sources} Reconstructions of the kite-shaped scatterer with different number of incident sources by using the smooth sawtooth wave $\chi_2$, where the black points denotes the locations of  receivers and the red small circles denotes the locations of the incident sources.    }
\end{figure}

\begin{figure}
\centering
\subfigure[geometry setting ]{\includegraphics[width=0.45\textwidth]
                   {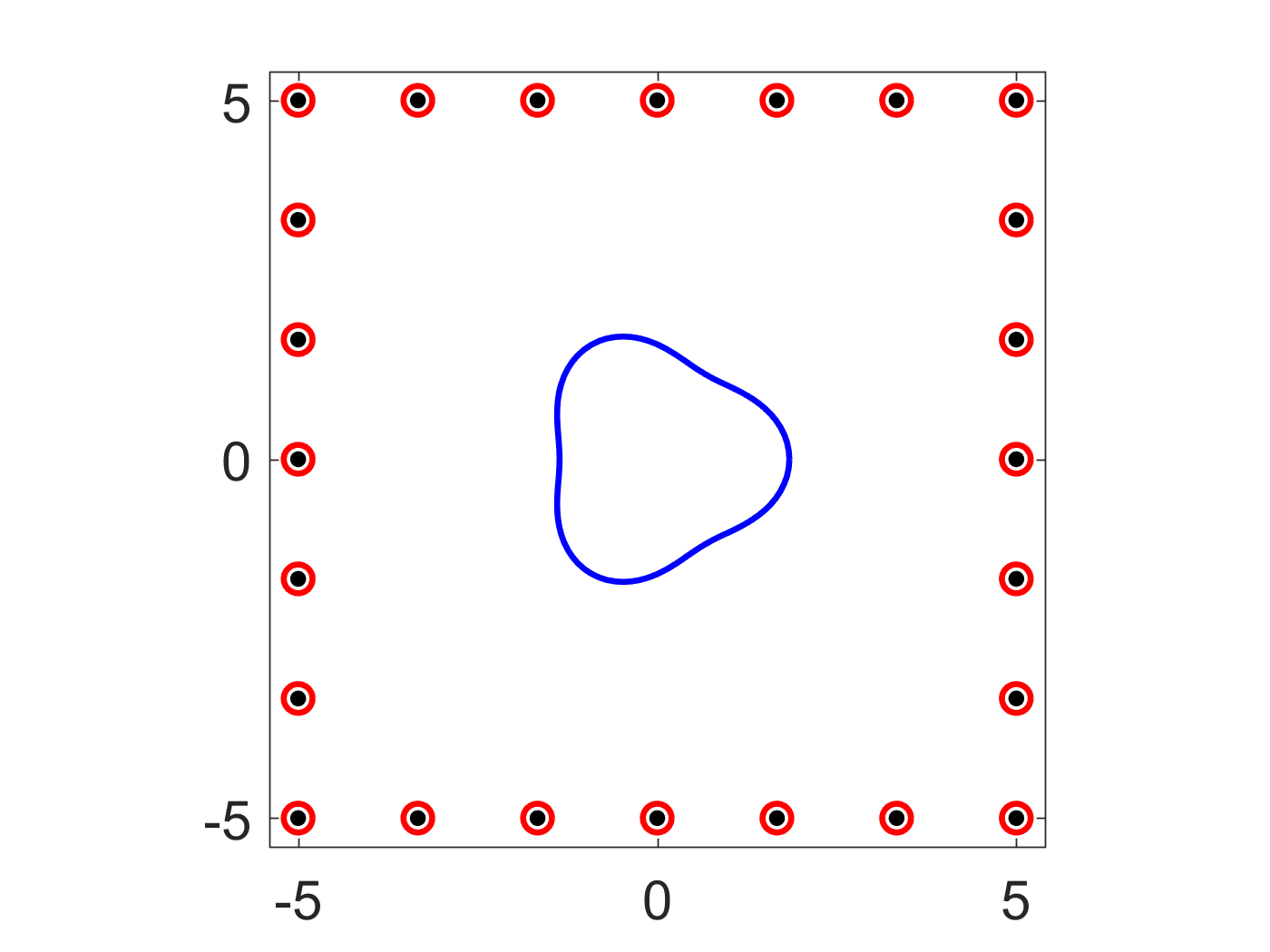}}
\subfigure[$\sigma=0$]{\includegraphics[width=0.45\textwidth]
                   {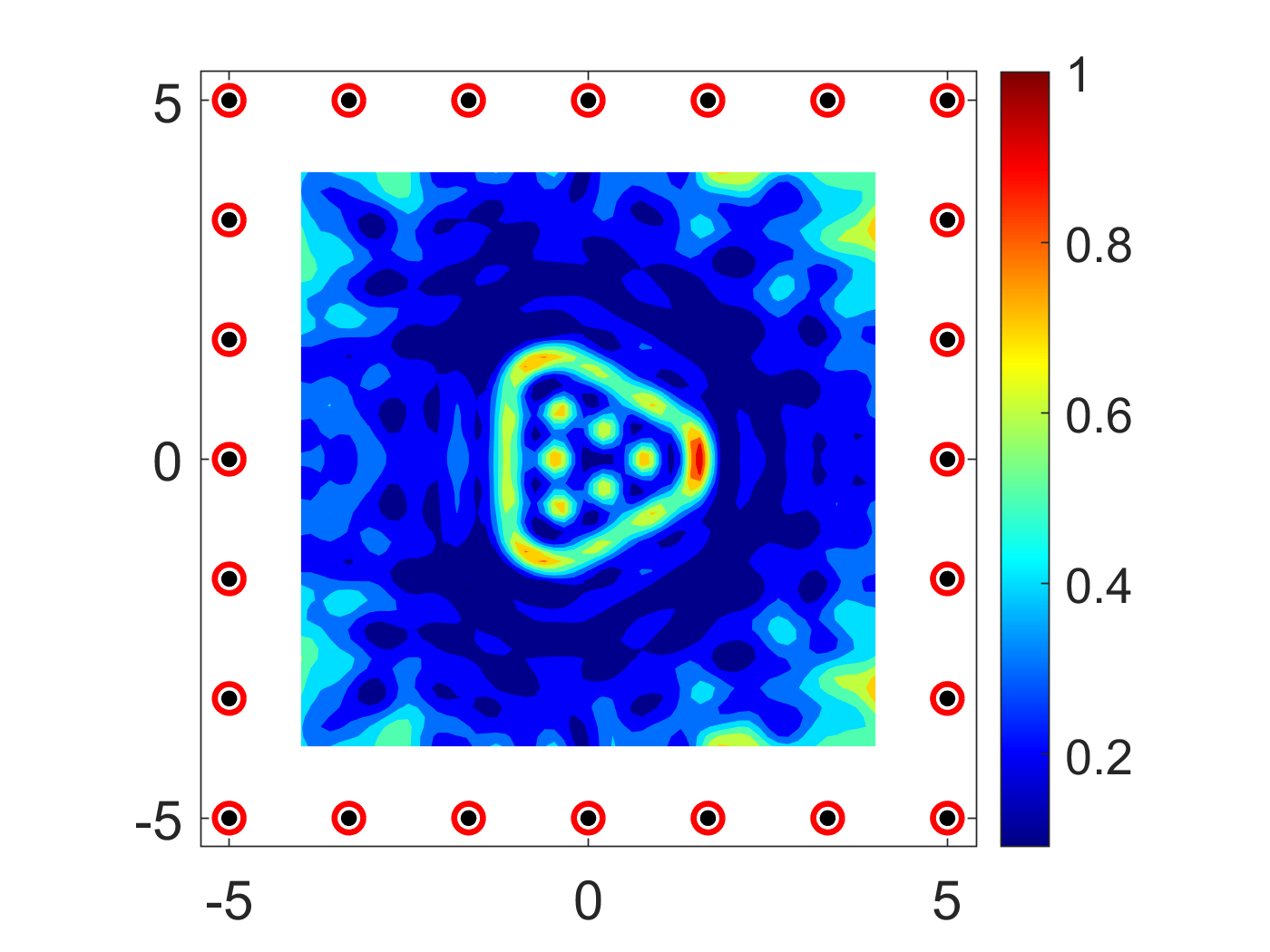}}\\
 \subfigure[$\sigma=1$]{\includegraphics[width=0.45\textwidth]
                   {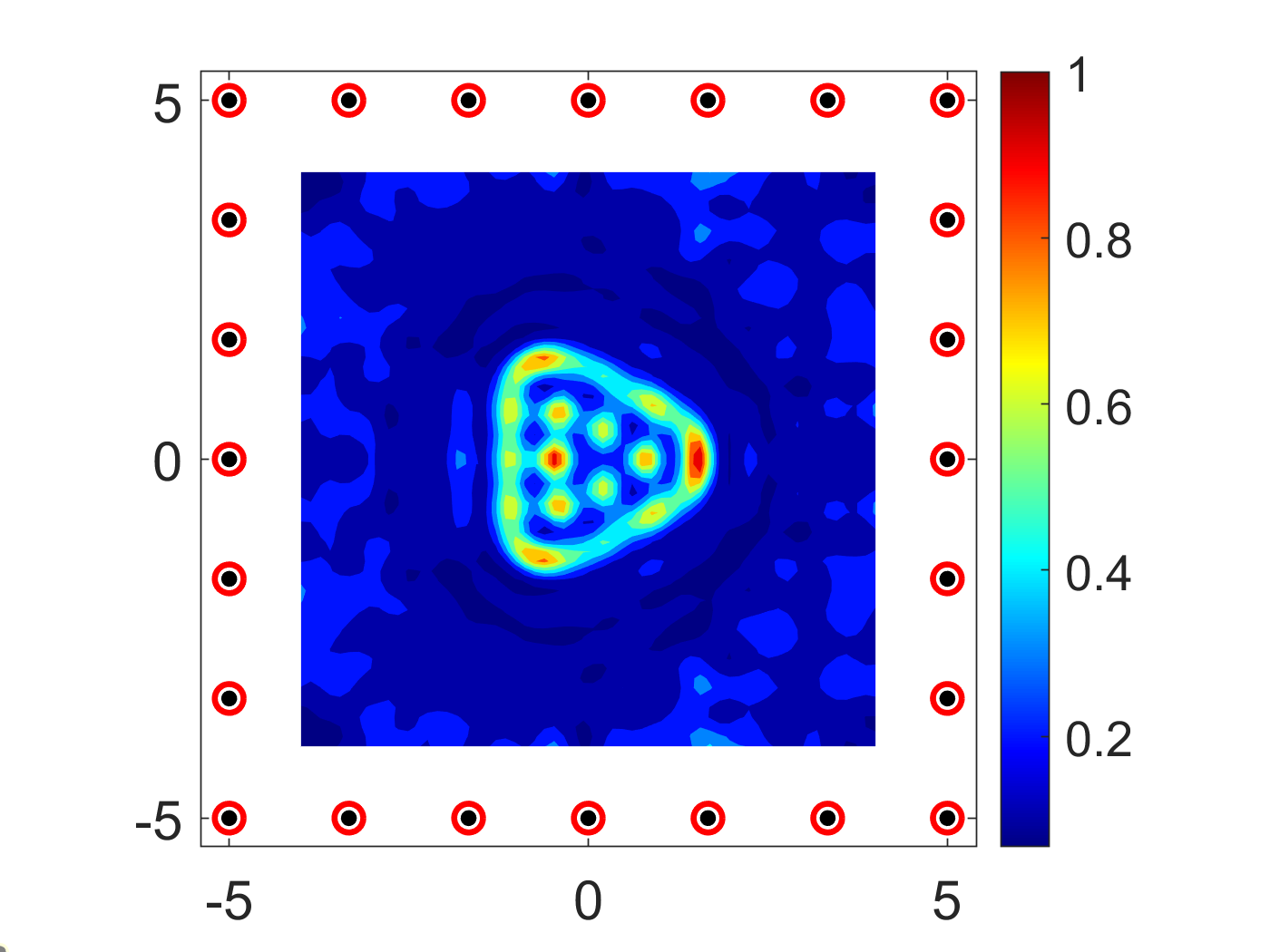}}
 \subfigure[$\sigma=2$]{\includegraphics[width=0.45\textwidth]
                   {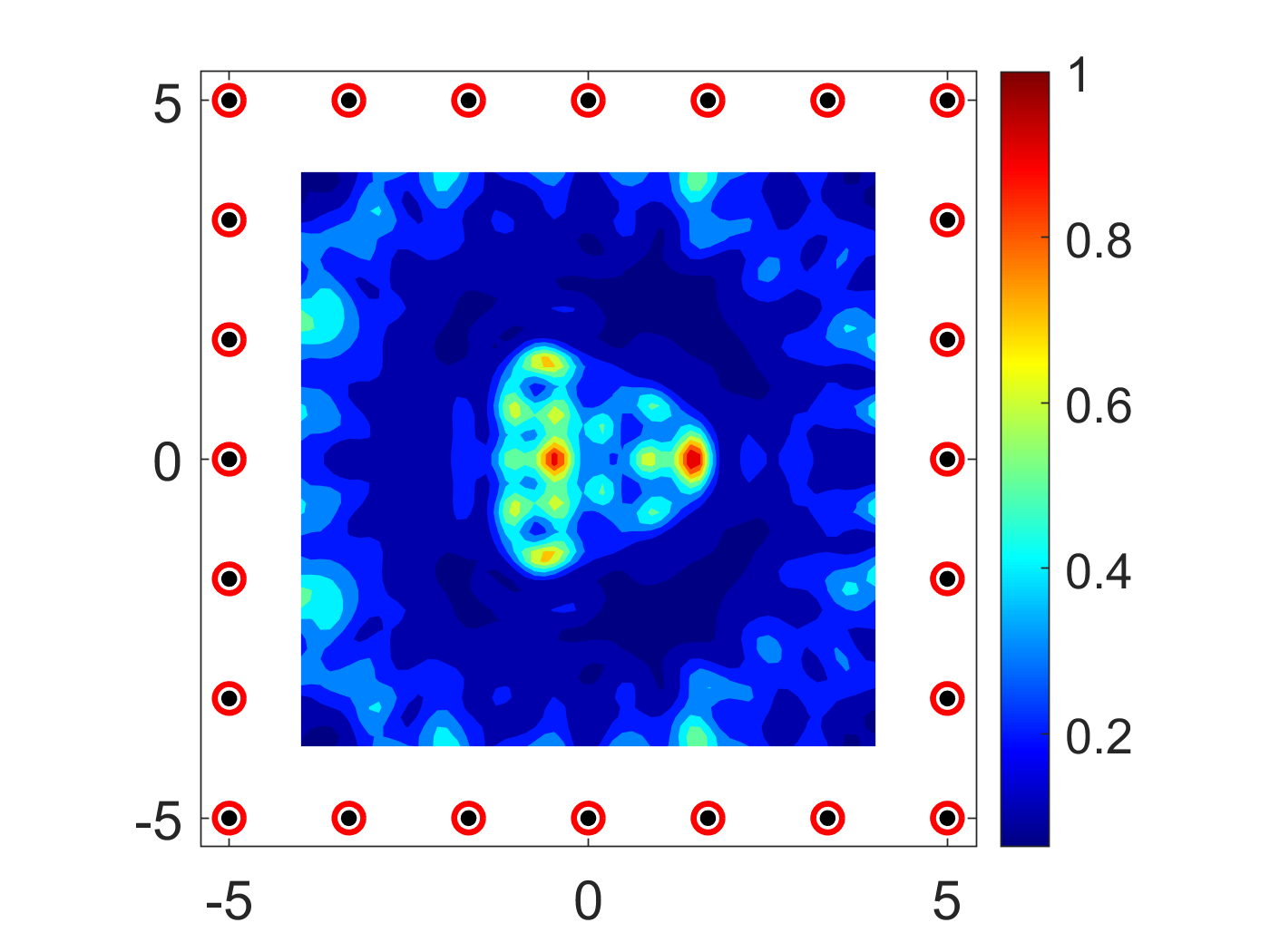}}
\caption{\label{fig:pear} Exact and reconstructed pear-shaped scatterer with different parameters $\sigma$ by using the tempered sinusoidal wave $\chi_3$. }
\end{figure}


\section*{Acknowledgement}

The work of Xianchao Wang was supported by the NSFC grant under No. 12001140.
 The work of Yukun Guo was supported by NSFC grants under No. 11971133.


\end{document}